\documentclass[a4paper,11pt,draft]{article}
\usepackage[utf8]{inputenc}
\usepackage{makeidx}
\usepackage[english]{babel}
\usepackage{amsthm}
\usepackage{amsmath}
\usepackage{latexsym}
\usepackage{amsfonts}
\usepackage{amssymb}
\usepackage{txfonts}
\usepackage{pxfonts}
\usepackage{enumerate}
\usepackage[all]{xy}
\usepackage{color}
\usepackage[notref,notcite]{showkeys}

\theoremstyle{plain}
\newtheorem{theorem}{Theorem}[section]
\theoremstyle{plain}
\newtheorem{prop}[theorem]{Proposition}
\theoremstyle{plain}
\newtheorem{coro}[theorem]{Corollary}
\theoremstyle{plain}
\newtheorem{lemma}[theorem]{Lemma}
\theoremstyle{definition}
\newtheorem{defi}[theorem]{Definition}
\theoremstyle{definition}

\theoremstyle{definition}
\newtheorem{example}[theorem]{Example}
\theoremstyle{remark}
\newtheorem{remark}[theorem]{Remark}
\theoremstyle{definition}
\newtheorem*{question}{Question}

\newcommand{\sbz}[1]{\Sigma B\mathbb{Z}/#1}
\newcommand{\bz}[1]{B\mathbb{Z}/#1}
\newcommand{\z}[1]{\mathbb{Z} #1}

\newcommand{\map}{\mbox{map}}

\newcommand{\Q}{\mathbb{Q}}
\newcommand{\Hom}{\operatorname{Hom}}
\newcommand{\p}[2]{#1^{\wedge}_{#2}}
\newcommand{\mapp}{\mbox{map}_{\ast}}

\newcommand{\Am}{B_m}
\newcommand{\An}[1]{B_{#1}}
\newcommand{\Amm}[1]{B_{m_{#1}}}

\newcommand{\SAm}{\Sigma B_m}
\newcommand{\SAn}[1]{\Sigma {B_{#1}}}
\newcommand{\Ima}{\operatorname{Im}}
\newcommand{\Ext}{\operatorname{Ext}}

\newcommand{\F}{\mathcal{F}}

\newcommand{\Inn}{\operatorname{Inn}}

\newcommand{\Lin}{\mathcal{L}}
\newcommand{\Linp}{B\mathcal{F}}
\newcommand{\Mor}{\operatorname{Mor}}
\newcommand{\Syl}{\operatorname{Syl}}

\newcommand{\OF}{\mathcal{O(F)}}
\newcommand{\OFc}{\mathcal{O}^{c}(\mathcal{F})}

\newcommand{\Rep}{\operatorname{Rep}}
\newcommand{\Topo}{\operatorname{Top}}
\newcommand{\hocolim}{\operatorname{hocolim}}

\newcommand{\Qp}{\hat{\Q}_p}

\title{Cellular approximations of $p$-local compact groups}
\author{Nat\`alia Castellana\footnote{The three authors are partially
supported by FEDER-MEC grant MTM2010-20692 and by the grant UNAB10-4E-378 co-funded by FEDER.}, Ram\'on Flores\footnote{The second author is partially supported by FEDER-MEC grant MTM2016-76453-C2-1-P.} and Alberto Gavira-Romero\footnote{The third author is partially supported by Proxecto Emerxente da Xunta, EM 2013/016.}}

\date{}

\makeindex

\begin{document}

\maketitle

\begin{abstract}

Let $A$ be the classifying space of an abelian $p$-torsion group. We compute $A$-cellular approximations (in the sense of Chach\'olski and Farjoun) of classifying spaces of $p$-local compact groups, with special emphasis in the cases which arise from honest compact Lie groups.

\end{abstract}

\section{Introduction}
\label{section1}

In 1949, J.H.C. Whitehead introduced the concept of CW-complex as a method to construct every topological space by means of spheres. This can be considered one of the most influential notions ever defined in Homotopy Theory, and in particular it allowed to describe the structure of the spaces making use of induction over the cellular skeleta. Following this philosophy, Farjoun and Chach\'olski published a series of papers in the nineties (see \cite{MR1392221} and \cite{MR1408539}, for example) in which they proposed to generalize the concept of cellular approximation. More precisely, given spaces $A$ and $X$, they say that a space is $A$-cellular if $X$ can be built from $X$ by means of (pointed) homotopy colimits. Roughly speaking, this means that $X$ is $A$-cellular if it can be constructed using copies of $A$ as building bricks. Moreover, they also define the $A$-cellular approximation $CW_AX$ (see Proposition \ref{p:CW def} below) of a space $X$ as the closer $A$-cellular space to $X$ in a certain precise sense. Taking $A=S^0$, all these concepts specialize to classical Whitehead cellular notions, and cellular approximations generalize the theory of CW-approximations.

However, there is a crucial difference between CW-approximation and cellular approximation (also called cellularization in the literature). For every space $X$, the CW-approximation preserves the homotopy of $X$, while for a general $A$ this does not happen for $CW_AX$. This feature of the functors, that could be perceived at first glance as a disadvantage, allowed to interpret the cellularization as a tool to identify the part of the structure of the spaces that depend on the space $A$, and ultimately to classify spaces in terms of cellularity classes. This point of view is related to Hopkins-Smith classification of thick subcategories of finite spectra \cite{Hopkinsmith}, and has been recently adopted in many different categories, as for example groups \cite{MR2995665}, chain complexes \cite{MR2581213}, ring spectra \cite{MR2200850}, etc. In each case, an appropriate definition of limit is used.

In the unstable homotopy category a complete classification should not be expected, because this would imply a classification of the ideals in the homotopy groups of the spheres. However, when the interest is mod $p$ homotopy theory, much can be said about the $A$-cellular structure of spaces if $A$ is a mod $p$ meaningful space. For example, when $A$ is a Moore space $M(\mathbb{Z}/p,n)$, Rodr\'{i}guez-Scherer \cite{MR1851263} and Chach\'olski \cite{MR1408539} have described with precision the structure of the $A$-cellular spaces and highlighted the close relationship between this homotopy approximation and the corresponding $\pi_1A$-cellularization at the level of groups. From a somewhat different but related point of view, Chach\'olski-Farjoun-Flores-Scherer \cite{ChaFarFloSche} recently identified cellular approximations of nilpotent Postnikov pieces, in a paper that was originally intended to understand all the possible approximations of classiying spaces of finite $p$-groups.

Let $A$ be the classifying space of an abelian $p$-torsion group. The present paper is part of a long-term program designed to understand the $A$-cellular structure of classifying spaces of groups and their mod $p$ homotopical analogues. The possibility of such descriptions was suggested by the deep understanding of the structure of $\textrm{map}_*(B\mathbb{Z}/p,X)$ that, for a wide range of spaces $X$, followed Miller's resolution of the Sullivan conjecture \cite{MR750716} and subsequent work of Lannes \cite{MR1179079} and others. This kind of information and knowledge of the underlying group structure leads to descriptions of these cellular approximations for finite groups \cite{MR2823972}, Eilenberg-MacLane spaces \cite{MR2262843}, infinite loop spaces \cite{CellInfEsp} and compact Lie groups \cite{CASTFLOR}.

The main goal of our paper has been to improve the mentioned results about approximation of compact Lie groups, as well as to understand their cellular structure in the most general framework of $p$-local compact groups. In a series of papers (\cite{MR1992826}, \cite{MR2222502}, \cite{MR2302494}) Broto-Levi-Oliver have defined the ``correct" mod $p$ homotopical generalizations of classifying spaces of finite groups and compact Lie groups, called respectively $p$-local finite groups and $p$-local compact groups.  As the concept of cellular approximation is essentially of a homotopic nature, it was a natural question to translate and/or generalize our results concerning $CW_ABG$ to this more general setting. In this sense, our main result for $p$-local compact groups has been the following (the case of $p$-local finite groups is treated in a separate paper \cite{CellFusSys}):

\vspace{.5cm}
\textbf{Theorem} (\ref{p:BS cell CW A Xp = Xp} and \ref{c:BS cell CW A X fib rac}). Let $(S,\F)$ be a $p$-local compact group with $H_2(\Linp;\z{})$ finite. Let $B_m=B\mathbb{Z}/{p^{\infty}}\times B\mathbb{Z}/{p^m}$.

If $BS$ is $B_m$-cellular then the augmentation map $c_{\Linp} \colon CW_{B}(\Linp) \rightarrow \Linp$ is a mod $p$-equivalence. Moreover, there exists $n_{\mathcal F}$ such that $c_{\Linp}$ is a mod $p$-equivalence for all $n\geq n_{\mathcal F}$.

If $\Linp$ is nilpotent,
  $$ CW_{B_m}(\Linp) \rightarrow \Linp \rightarrow (\Linp)_{\Q}$$  is a homotopy fibration for all $m \geq n_{\mathcal F}$.

\vspace{.5cm}

A key tool to understand the result is Notbohm's kernel of a map from a classifying space (see Definition \ref{d:kerf} below, and the original reference \cite{MR1286829} for a thorough account). This group is defined in a homotopical way and play in our context the role of the minimal strongly closed subgroup that contains the $p^m$-torsion in the finite case.


The knowledge of the role of the Notbohm kernel allows to describe also the $B\mathbb{Z}/p^m$-cellular structure of $BG$, for $G$ a compact connect Lie group and for \emph{every} $m>0$. Note that in this way we are able to isolate the homotopy structure of $\p{BG}{p}$ that does \emph{not} depend on $B\mathbb{Z}/{p^{\infty}}$, and in particular we answer the final question of \cite{CASTFLOR} for these groups:

\vspace{.5cm}

 \textbf{Theorem} (\ref{t:CW Bzpm BG}).
  Let $m \geq 0$, and $G$ be a compact connected Lie group. Then there exists a compact Lie group $H$ and a map $\pi \colon \p{BG}{p}  \rightarrow \p{BH}{p}$ whose homotopy fibre is
  $\p{B\Gamma}{p}$, where $\Gamma$ is a compact Lie group, such that $CW_{\bz{p^m}}(\p{BG}{p})$ is homotopy equivalent to the homotopy fibre of the rationalization
  $\p{B\Gamma}{p} \rightarrow (\p{B\Gamma}{p})_{\Q}$.

\vspace{.5cm}

It is noteworthy that the classification of the simple 1-connected simple Lie groups allow to identify the group $H$ for any choice of the connected group $G.$

 Recall that for obtaining the complete description  of $CW_{B\mathbb{Z}/p}BG$ for $G$ finite, it was necessary to undertake the task of classifying the strongly closed subgroups of the finite groups \cite{MR2538017}. In the compact case, when $G$ is moreover simple, Notbohm was able to complete the corresponding classification, and that allowed us to be more precise in this case:

\vspace{.5cm}

 \textbf{Proposition} (\ref{p:G simple}).
  Let $G$ be a compact $1$-connected simple Lie group,  $p$ be a prime such that $p \mid |W_G|$. Then for all $m \geq 1$, the $\bz{p^m}$-cellularization of $\p{BG}{p}$ is
  equivalent to the homotopy fibre of the rationalization $\p{BG}{p} \rightarrow (\p{BG}{p})_{\Q}$, unless $G = S^3$ and $p^m = 2$. In this case $CW_{B\mathbb{Z}/2}BS^3=B\mathbb{Z}/2$.

\vspace{.5cm}

 Note that the case in which the prime $p$ does not divide the order of the Weyl group was already solved in \cite{CASTFLOR}.

 The paper is structured as follows: Section \ref{section2} introduces cellular spaces and $p$-local compact groups, with special emphasis in the notion of Notbohm kernel; in Section \ref{section3} we study the cellular structure of discrete $p$-toral groups, which play the role of the Sylow subgroups in the setting of $p$-local compact groups; in Section \ref{section4} the main general results about cellularization of $p$-local compact groups are proved; Section \ref{section5} contains the consequences about the cellular structure of classifying spaces of compact Lie groups, and we finished in last section with the sharpened results for compact connected Lie groups and other examples. In particular, we compute the cellular approximation of $BDI(4)$, an exotic example of $p$-local compact group.

\medskip

\noindent{\textbf {Notation:}} Given $m\geq 1$, the classifying space $B\z /p^{\infty}\times B\z /p^m$ will be denoted $B_m$. By convention we will assume $B_0=B\z /p^{\infty}$.




\section{Preliminaries}
\label{section2}

 In this section we collect the main definitions about cellular spaces and $p$-local compact groups that will be intensively used in the remaining of the paper. As these topics have been previously developed with detail (see references below), we only intend to offer a brief account to the theory, and we refer the reader to the  sources below for a more thorough account to the subject.

\subsection{Cellular spaces}

The main available references about cellular spaces are the book \cite{MR1392221} of Farjoun and the paper \cite{MR1408539} of Chach\'olski.
These authors introduced the concept of cellular space in order to generalize CW-complexes to a more general setting called $A$-homotopy theory. Let then $A$ and $X$ be pointed spaces.

\begin{defi}
A space $X$ is said $A$-\emph{cellular} if it can be constructed out of $X$ by means of (iterated) pointed homotopy colimits. Moreover, a class of pointed spaces is called cellular if it is closed under weak equivalences and homotopy colimits.
\end{defi}

 In this sense, the $A$-cellular spaces are the elements of the smallest closed class $\mathcal{C}(A)$ that contains $A$, and it is said that $A$ \emph{generates} $\mathcal{C}(A)$. If $X\in \mathcal{C}(A)$, the notation $A<<X$ is usual, and this is called a \emph{cellular inequality}. A good compilation of cellular inequalities can be found in \cite[Section 3]{ChaFarFloSche}.


The basic criterion to recognize $A$-cellular spaces is the following result.

\begin{theorem}[{\cite[Section 5]{MR1408539}}]
A pointed space $X$ is $A$-cellular if and only if for any map $f:Y\rightarrow Z$ such that the induced map $\mapp (A,Y)\rightarrow \mapp (A,Z)$ is a weak equivalence, it is true that the induced map $\mapp (X,Y)\rightarrow \mapp (X,Z)$ is also a weak equivalence.
\end{theorem}

E. Farjoun also showed the existence of a functor such that, given a pointed space $A$, turns every pointed space $X$ to an $A$-cellular space, and which in fact offers the best possible approximation for the point of view of maps coming out of $A$.

\begin{prop}[{\cite[Section 2.E]{MR1392221}}]
\label{p:CW def}
There exists an augmented  idempotent functor $CW_A:\mathbf{Spaces}_*\rightarrow \mathbf{Spaces}_*$ such that $CW_AX$ is $A$-cellular for every $X$. The functor $CW_A$ preserves weak equivalences and for every $X\in \mathbf{Spaces}_*$, the augmentation induces a map  $\mapp (A,CW_AX)\rightarrow \mapp (A,X)$ which is a weak equivalence. Moreover a space $X$ is $A$-cellular if and only if the augmentation $CW_AX\rightarrow X$ is itself a weak equivalence.
\end{prop}

The space $CW_AX$ is called the $A$-\emph{cellularization} or $A$-\emph{cellular approximation} of $X$, while the map $CW_AX\rightarrow X$ is called the $A$-\emph{cellular cover}. It is also common to denote $cell_AX$ instead of $CW_AX$.

A disadvantage of Farjoun construction of $CW_AX$ is that it includes a transfinite process which makes very difficult to figure out $CW_AX$ out of $X$. This problem was overcame by the following result of Chach\'olski, which is maybe the most powerful available tool to compute cellular approximations, and which is crucial in this paper.

Recall that for every $X$, the \emph{A-nullification} of $X$ is the localization of $X$ with respect to the constant map $A\rightarrow \ast$ (see \cite{MR1408539}, Section 14) denoted by $P_AX$. We say $X$ is $A$-null if $ev\colon \map(A,X)\rightarrow X$ is a homotopy equivalence ($X\simeq P_AX$) and $X$ is $A$-acyclic if $P_AX\simeq *$. From the definitions we have that $A$-cellular spaces are $A$-acyclic.

\begin{theorem}[{\cite[Theorem 20.3]{MR1408539}}]
\label{t:chacholski}
Let $A$ be pointed and connected and let $X\rightarrow X'$ be a pointed map.  Assume that the homotopy fibre of this map is $A$-cellular, and the induced
map $[A,X]_*\rightarrow [A,X']_*$ is trivial. Then the map from the homotopy fiber of the composition $(X\rightarrow X'\rightarrow P_{\Sigma A}X')$ to $X$ is the $A$-cellular cover of $X$; that is, there is a homotopy fiber sequence
$$CW_AX\rightarrow X\rightarrow P_{\Sigma A}X'.$$
\end{theorem}

Observe in particular that the previous theorem implies that the fibration $CW_AX\rightarrow X$ is principal, and in particular the functor $CW_A$ preserves nilpotence \cite[V.5.2]{MR0365573}.

\begin{remark}\label{r:chacholski}
When $A$ and $X$ are understood and a choice of $X'$ is made, the map $X\rightarrow X'\rightarrow P_{\Sigma A}X'$ will be called in the text \emph{Chach\'olski fibration}. The most common choice of $X'$ is the $\Sigma A$-nullification of the homotopy cofibre of the map $\bigvee_{[A,X]_*}A\rightarrow X$, where the wedge is extended to all the pointed homotopy classes of maps $A\rightarrow X$ and the map over the wedge is defined in the obvious way ( see \cite[Section 20]{MR1408539} for more information).
\end{remark}

\begin{remark}\label{r:wedgevsproduct}
 Note that $X<<Y$ if and only if $\mathcal{C}(Y)\subset \mathcal{C}(X)$. Then if $\mathcal{C}(Y)=\mathcal{C}(X)$ we have that the functors $CW_X$ and $CW_Y$ are naturally equivalent. In particular, this applies to the wedge and product of two spaces, $\mathcal{C}(A\vee B)=\mathcal{C}(A\times B)$.
 \end{remark}

We finish with a useful result which describes behavior of cellularity under fibrations.

\begin{prop}[{\cite[3.E.1]{MR1392221}}]
\label{p:CW fib}

{\color{black} Let $A$ be a space, $F\rightarrow E\rightarrow B$ be a fibration of pointed spaces. If $F$ and $E$ are $A$-cellular, then $B$ is so.}

\end{prop}

\subsection{$p$-local compact groups}

We will follow here the source \cite{MR2302494}, and we refer the reader to it for any additional information he/she needs.

 In this section we want to introduce the concept of ``$p$-local homotopy theory'' of classifying spaces of compact Lie groups, that is, the homotopy theory of the $p$-completion of those classifying spaces.
 If $G$ is a compact Lie group, then it turns out that there is a close connection between the $p$-local homotopy theory of $BG$ and the ``$p$-local structure'' of $G$, this means,
 the conjugacy relations in a Sylow $p$-subgroup of $G$ (i.e., in a maximal $p$-toral group). This connection then suggested the construction of certain spaces which have many of the
 same properties as $p$-completed classifying spaces of a compact Lie group have: the classifying spaces of ``$p$-local compact groups''.

 First we need to describe discrete $p$-toral groups, which in particular play the role of Sylow $p$-groups in $p$-local compact groups. We will give certain detail here, as the cellular strucure of their classifying spaces will be crucial for our further developments.

 \begin{defi}\label{d:p toral group}
  A \textit{$p$-toral group} is a compact Lie group $P$ which contains a normal subgroup $P_0 \trianglelefteq P$ isomorphic to a torus, i.e, $P_0 \cong (S^1)^r$, and such that the
  quotient $P/P_0$ is a finite $p$-group.

  For such a group, we say that $P_0$ is the \textit{connected component} or \textit{maximal torus} of $P$ and $P/P_0$ is the \textit{group of components} of $P$.

  The \textit{rank} of a $p$-toral group $P$ is the rank of the maximal torus $P_0$. That is, if $P_0 \cong (S^1)^r$, then $rk(P) = r$.
 \end{defi}

 We have that a $p$-toral group is an extension of a torus and a finite $p$-group; analogously, a discrete $p$-toral group is an extension of a discrete $p$-torus, that is, a finite product of copies of $\z{/p^{\infty}}$ and a finite $p$-group.

 \begin{defi}\label{d:disc p toral group}
  A \textit{discrete $p$-toral group} is a group $P$ which contains a normal subgroup $P_0 \trianglelefteq P$ isomorphic to $(\z{/p^{\infty}})^r$, and such that the quotient
  $P/P_0$ is a finite $p$-group. In this case, the \textit{rank} of $P$ is $rk(P) = r$, the \textit{connected component} or \textit{maximal torus} of $P$ is $P_0$, and $P/P_0$ is
  called the \textit{group of components} of $P$. If $P = P_0$, then $P$ is called \textit{connected}.
 \end{defi}

 For example,  a finite $p$-group is a discrete $p$-toral group of rank $0$.

We are now ready to define the concept of $p$-local compact group as formalized by Broto-Levi-Oliver \cite{MR2302494}, that built in turn on previous work of Puig \cite{MR2253665} for the finite case.  Such objects have classifying spaces which satisfy many of the homotopy theoretic properties of $p$-completed classifying spaces of finite groups and compact Lie groups.

A saturated fusion system (see the formal definition in \cite[2.1, 2.2]{MR2302494}) is a category whose objects are the subgroups of $S$, and
whose morphisms are injections between these subgroups. These morphisms are intended to codify the \emph{fusion data} of the $p$-local compact group,
in the same sense as the injections and conjugations codify the fusion in a finite or more generally in a $p$-compact group. The centric linking system
(formal definition in \cite[4.1]{MR2302494}) is another category whose objects are some subgroups of $S$ (the \emph{centric} subgroups \cite[2.6]{MR2302494}),
and such that if $\F^c$ is the subcategory of $\F$ whose objects are these subgroups, there is a projection functor $\pi:\Lin \rightarrow \F^c$ which is surjective
on morphisms. It is noteworthy that the original definition of Broto-Levi-Oliver of $p$-local compact group included the linking system in the definition,
but this has been proved redundant by a recent result of Levi-Libman \cite{MR3338608}, Chermak \cite{Chermak} and Oliver \cite{OliverUniq}.

 \begin{defi}
  A \textit{$p$-local compact group} is defined to be a pair $(S,\F)$, where $S$ is a discrete $p$-toral group (called the $p$-Sylow) and $\F$ is a
  \emph{saturated fusion system} over $S$. Associated to $\F$ there is a unique centric linking system $\Lin$ whose $p$-completed nerve $\p{|\Lin|}{p}$ is called the \textit{classifying space} of the $p$-local compact group and usually denoted $\Linp$. If $S$ is a finite $p$-group, hence $(S,\F)$ is called a \textit{$p$-local finite group}.
 \end{defi}

Of course, every compact Lie group gives rise in every prime $p$ to a $p$-local compact group, and the nerve of $\Lin$ plays the role of the classifying space of the group in that prime:

 \begin{theorem}[{\cite[Theorem 9.10]{MR2302494}}]
  Fix a compact Lie group $G$ and a maximal discrete $p$-toral subgroup $S<G$. Then there exists a saturated fusion system $\F_S(G)$ associated to group $G$ and a centric linking system $\Lin^c_S(G)$ associated to $\F_S(G)$ such that
  $(S, \F_S(G))$ is a $p$-local compact group with classifying space $\p{|\Lin_S(G)|}{p} \simeq \p{BG}{p}$.
 \end{theorem}

 One of the standard techniques used when studying maps between $p$-completed classifying spaces of finite groups is to replace them by the $p$-completion of a homotopy colimit of
 simpler spaces.

 \begin{defi}
  Let $\F$ be a saturated fusion system over a discrete $p$-toral group $S$. The \textit{orbit category} of $\F$ is the category $\OF$ whose objects are the subgroups of $S$, and
  whose morphisms are defined by
  $$
   \Mor_{\OF}(P,Q) = \Rep_{\F}(P,Q) := \Inn(Q)/\Hom_{\F}(P,Q).
  $$
  We let $\OFc$ denote the full subcategory of $\OF$ whose objects are the $\F$-centric subgroups of $S$. If $\Lin$ is a centric linking system associated to $\F$, then
  $\tilde{\pi}$ denotes the composite functor
  $$
   \xymatrix@1{\tilde{\pi} \colon \Lin \ar[r]^-{\pi} & \F^c \ar@{->>}[r] & \OFc}
  $$
 \end{defi}

 We next look at the homotopy type of the nerve of a centric linking system.

 \begin{prop}\label{p:L = hocolim}
  Fix a saturated fusion system $\F$ over a discrete $p$-toral group $S$ and an associated centric linking system $\Lin$, and let $\tilde{\pi} \rightarrow \OFc$ be the projection
  functor. Let
  $$
   \tilde{B} \colon \OFc \rightarrow \Topo
  $$
  be the left homotopy Kan extension over $\tilde{\pi}$ of the constant functor $\xymatrix@1{\Lin \ar[r]^-{\ast} & \Topo}$. Then $\tilde{B}$ is a homotopy lifting of the homotopy
  functor $P \mapsto BP$, and
  $$
   |\Lin| \simeq \hocolim_{\OFc} (\tilde{B}).
  $$
 \end{prop}

 Finally, we remark an important class of subgroups of $S$ that are preserved by fusion:

 \begin{defi}
  Let $\F$ be a fusion system over a discrete $p$-toral $S$. Then a subgroup $K \leq S$ is \textit{strongly $\F$-closed} if for all $P \leq K$ and all morphism
  $\varphi \colon P \rightarrow S$ in $\F$ we have $\varphi(P) \leq K$.
 \end{defi}

 Note that if $G$ is a compact Lie group and $S \in \Syl_p(G)$,  $K \subseteq S$ is strongly $\F_S(G)$-closed if and only if $K$ is \textit{strongly closed in $G$}, i.e., if for all
 $k \in K$ and $g \in G$ such that $c_g(s) \in S$, then $c_g(s) \in K$.

 It is easy to see that an intersection of strongly $\F$-closed subgroups is a strongly $\F$-closed subgroup. Hence note that, given a subgroup $P \leq S$, we denote by $Cl_{\F}(P)$
 the smallest strongly $\F$-closed subgroup of $S$ that contains $P$. This definition makes sense since discrete $p$-toral subgroups are artinian \cite[Proposition 1.2]{MR1992826}. We use the notation $Cl_{\F}(\z /p^m)$ for $Cl_{\F}(\Omega_{p^m}S)$, where $\Omega_{p^m}S$ is the normal subgroup of $S$ generated by all elements of order less than or equal to $p^m$.

\subsection{Notbohm kernel of maps from classifying spaces}

A key ingredient in \cite{MR2823972} which allowed to describe the $\bz{p}$-cellularization of $BG$ for $G$ finite was a precise description of the minimal strongly closed subgroup $Cl(S)$ of the $p$-Sylow $S<G$ that contains all the order $p$-elements of $S$, and a group-theoretic classification of these objects. Then, in order to understand the cellularization of more general classifying spaces, it was necessary to identify a subgroup of the corresponding Sylow $p$-subgroup which played the role of $Cl(S)$ in these broader contexts. In the context of fusion systems, this role will not be played by $Cl_{\F}(\z /p^m)$, but by a bigger subgroup of $S$: the \emph{Notbohm kernel} of a certain map.



 \begin{defi}\label{d:kerf}
  Let $(S,\F)$ be a $p$-local compact group and $Z$ be a $p$-complete and $\sbz{p}$-null space. Let $f \colon \Linp \rightarrow Z$ be a pointed map. Then the \emph{Notbohm kernel} of $f$ is defined as
  $$
   \ker(f) := \{g \in S \mid f|_{B\langle g \rangle} \simeq \ast \}.
  $$
 \end{defi}

 Notbohm proposed his definition of kernel in the context of compact Lie groups, and the original definition is used by the first and third author in \cite{CellFusSys}, where they explore cellular approximations for $p$-local finite groups. The generalization we propose to $p$-local compact groups is nevertheless immediate.

The following is a concrete description of the Notbohm kernel in the case of a map between $p$-completed classifying spaces of compact Lie groups. For simplicity, the Notbohm kernel will be called simply `kernel' in the remaining of the paper, as long as there will be no confusion with any other kernel of a group homomorphism.

 \begin{example}
  Let $f \colon \p{BH}{p} \rightarrow \p{BG}{p}$ be a map where $H$ and $G$ are compact Lie groups. Let $S \in \Syl_p(H)$, then for the map
  $f|_{BS} \colon BS \rightarrow \p{BH}{p} \rightarrow \p{BG}{p}$ there exists a group homomorphism $\rho \colon S \rightarrow G$ such that $f|_{BS} \simeq B\rho$ according to
  \cite[Theorem 1.1']{MR928826}. In this case, $\ker(f)$ is the kernel of the group homomorphism $\rho$.
 \end{example}

 The kernel of a map from the classifying space of a $p$-local compact group verifies similar properties to the kernel of a map between $p$-local finite groups. These properties will be crucial to describe the role of the kernel in the cellular structure of the classifying spaces we are interested on, compare with \cite[Section 4]{MR2351607}.

 \begin{prop}\label{p:kerf strongly}
  Let $(S,\F)$ be a $p$-local compact group and let $f \colon \Linp \rightarrow Z$ be a pointed map, where $Z$ is a $p$-complete and $\sbz{p}$-null space. Then $\ker(f)$ is a
  strongly $\F$-closed subgroup of $S$.
 \end{prop}

 \begin{proof}
  See \cite[Proposition 3.4]{CellFusSys}.
 \end{proof}

Observe in particular that minimality of $Cl_{\F}(\z /p^m)$ implies that $Cl_{\F}(\z /p^m)\leq \ker(f)$.

 \begin{prop}\label{p:Dwyer p local}
  Let $(S,\F)$ be a $p$-local compact group and $Z$ be a $p$-complete and $\sbz{p}$-null space. Then a map $f \colon \Linp \rightarrow Z$ is null-homotopic if and only if
  $\ker(f) = S$.
 \end{prop}

 \begin{proof}

  The proof is identical to that of  \cite[Theorem 3.5]{CellFusSys}.
 \end{proof}

\section{Cellularization of classifying spaces of discrete $p$-toral groups}\label{s:CW cl disc p toral gp}
\label{section3}

In this paper we are mainly concerned with cellular approximation with respect to classifying spaces of $p$-torsion groups. It is not surprising then that if the target of these functors are classifying spaces of groups, their maximal $p$-torsion subgroups play a crucial role in the shape of the approximations. For example, it was proved in \cite{MR2351607} that if $S<G$ is a $p$-Sylow subgroup of a finite group, the ${\bz{p}}$-cellularity of $BS$ implies the ${\bz{p}}$-cellularity of $BG$ (assuming $G$ generated by order $p$ elements), and same happens for $S$ the maximal discrete $p$-toral subgroup of a compact Lie group $G$. .

If now we move to the realm of $p$-local compact groups, it is clear from the previous discussion that a necessary step to understand their cellularizations should be a description to some extent of the cellularization of discrete $p$-toral groups. For the case $A={\bz{p}}$, this task was undertaken in \cite[Example 6.16]{CASTFLOR} and it is not difficult to extend this result to
 $\bz{p^m}$-cellularization for $1 \leq m \leq \infty$. Recall that $\Omega_{p^m}P$ is the normal subgroup of $P$ generated by all elements of order less than or equal to $p^m$.


 \begin{prop}\label{p:CW BZpinfty BPinfty}
  Let $P$ be a discrete $p$-toral group with a maximal torus $P_0 \cong (\z{/p^{\infty}})^r$. Then,
  \begin{enumerate}[(a)]
   \item $CW_{\bz{p^{\infty}}}(BP) \simeq BP_0$,
   \item $CW_{\bz{p^m}}(BP) \simeq B \Omega_{p^m} P$.
  \end{enumerate}
 \end{prop}

 \begin{proof}
  \begin{enumerate}[(a)]
   \item
    Consider the homotopy fibration
    $
     \xymatrix@1@C=1.5em{BP_0 \ar[r]^-f & BP \ar[r] & B\pi}.
    $ where $\pi = P / P_0$ is the group of components of $P$.
    Since $\mapp(\bz{p^{\infty}},B\pi) \simeq \Hom(\z{/p^{\infty}}, \pi) = \{ e \}$, $B\pi$ is $\bz{p^{\infty}}$-null. Hence, $f$ is a $\bz{p^{\infty}}$-equivalence and so
    $CW_{\bz{p^{\infty}}}(BP) \simeq CW_{\bz{p^{\infty}}}(BP_0) \simeq BP_0$.
   \item
    The map $B \Omega_{p^m}(P) \rightarrow BP$ induced by the inclusion is a $\bz{p^m}$-equivalence since
    $
     \mapp(\bz{p^m}, BP) \simeq \Hom (\z{/p^m},P) = \Hom (\z{/p^m},\Omega_{p^m} P) \simeq \mapp(\bz{p^m}, B \Omega_{p^m}(P)),
    $
    and finally $B \Omega_{p^m} P$ is $\bz{p^m}$-cellular according to \cite[Corollary 2.5]{MR2262843}.
  \end{enumerate}
 \end{proof}

 The main goal of this section is to show that, given discrete $p$-toral group $P$, there exists an integer $m_0 \geq 0$  (which depends on $P$), such that $BP$ is cellular with respect to $\bz{p^{\infty}} \times \bz{p^m}$ for all $m \geq m_0$ (and then with respect to $\bz{p^{\infty}} \vee \bz{p^m}$, see Remark \ref{r:wedgevsproduct}).

Let $\pi$ be a finite group and $x \in \pi$ an element of order $|x| \leq p^m$. We denote by $\alpha_x \colon \z{/p^m} \rightarrow \pi$ the homomorphism defined by
 $\alpha_x(1) = x$, and by $f_x \colon \bz{p^m} \rightarrow B\pi$ the induced map. To establish the cellularity of $BP$, we need to understand before how maps from  $\bz{p^m}$ to the classifying space of the group of components lift to the classifying space of the $p$-toral group.


 \begin{lemma}\label{l:BZpm lifts BPinf}
  Let $P$ be a discrete $p$-toral group, $P_0 \cong (\z{/p^{\infty}})^r$ a maximal torus of $P$ and $\pi = P/P_0$ the group of components. For any element $x\in \pi$,
  there is an integer $m_x \geq 0$ and a map
  $\tilde{f_{x}} \colon \bz{p^{m_x}} \rightarrow BP$ such that the following diagram
  $$
   \xymatrix{                                                      & BP \ar[d]^-{Bpr}\\
             \bz{p^{m_x}} \ar[r]_-{f_{x}} \ar[ur]^-{\tilde{f_{x}}} & B\pi,
            }
  $$
  is commutative up to homotopy.
 \end{lemma}

 \begin{proof}
  Let $x \in \pi$ and $g \in pr^{-1}(x)$. If $Q = \langle q \rangle \leq P$, then $Q \cong \z{/p^{m_x}}$ for certain $m_x \geq 0$ since $P$ is a locally finite $p$-group
  (see \cite[Proposition 1.2]{MR2302494}). Now if $\beta_x \colon \z{/p^{m_x}}\rightarrow P$ is the homomorphism defined by $\beta_x(1) = g$, then the induced map $\tilde{f_x} \colon \bz{p^{m_x}} \rightarrow BP$ verifies the statement.
 \end{proof}


 Recall that for $m\geq 0$, the wedge $B\z /p^{\infty}\times B\z /p^m$ will be denoted $B_m$.

 \begin{prop}\label{p:BP Am cellular}
  Let $P$ be a discrete $p$-toral group. There is a non-negative integer $m_0$ such that $BP$ is $\Am$-cellular for all $m \geq m_0$.
 \end{prop}

 \begin{proof}
  We  follow the ideas in \cite[Proposition~2.4]{MR2262843}. Let $m$ be any non-negative integer and consider Chach\'{o}lski's fibre sequence
  $$
   \xymatrix@1@C=1.5em{CW_{\Am} (BP) \ar[r]^-{c} & BP \ar[r]^-r & P_{\SAm} C}.
  $$
  If $r \simeq \ast$, then $BP$ will be $\Am$-cellular because the fibre sequence implies a splitting
  $CW_{\Am} (BP) \simeq BP \times \Omega (P_{\SAm} C)$. Thus, applying the functor $P_{\Am}$, we get the chain of equivalences
  $$
   \ast \simeq P_{\Am} (CW_{\Am} (BP)) \simeq P_{\Am} (BP) \times P_{\Am}(\Omega (P_{\SAm} C)) \simeq P_{\Am} (BP) \times \Omega (P_{\SAm} C).
  $$
  Then $\Omega (P_{\SAm} C) \simeq \ast$ and, since $P_{\SAm} C$ is connected, $P_{\SAm} C \simeq \ast$.

 Next we construct a map $\tilde{r} \colon B\pi \rightarrow P_{\SAm} C$ such that
  $r \simeq \ast$ if and only if $\tilde{r} \simeq \ast$. In order to achieve this, we intend to apply Zabrodsky lemma to the fibration
  $\xymatrix@C=1.5em{BP_0 \ar[r]^-{B\iota} & BP \ar[r]^-{Bpr} & B\pi}$ and the map $r \colon BP \rightarrow P_{\SAm} C$.

  The following properties hold since $BP_0$ is $\bz{p^{\infty}}$-cellular, and then $\Am$-cellular:
  \begin{enumerate}[(a)]
   \item The space $\Omega (P_{\SAm} C) \simeq P_{\Am} (\Omega C)$ is $BP_0$-null, since $BP_0$ is $\Am$-acyclic:
         $$
          \mapp(BP_0, P_{\Am} (\Omega C)) \simeq \mapp(P_{\Am} (BP_0), P_{\Am} (\Omega C)) \simeq \ast.
         $$
   \item Since $BP_0$ is $\Am$-cellular, there exists a map $\tilde{B\iota} \colon BP_0 \rightarrow CW_{\Am} (BP)$ such that the following diagram is commutative up to homotopy
         $$
          \xymatrix{                          & BP_0 \ar[d]^-{B\iota} \ar[dl]_-{\tilde{B\iota}} \ar[dr]^-{r \circ B\iota} & \\
                    CW_{\Am} (BP) \ar[r]^-{c} & BP                                              \ar[r]^-r                 & P_{\SAm} C
                   }
         $$
         and therefore $r \circ B\iota \simeq r \circ c \circ B\iota \simeq \ast$.
  \end{enumerate}
  The hypotheses of Zabrodsky's Lemma are then satisfied and this implies the existence of a map $\tilde{r} \colon B\pi \rightarrow P_{\SAm} C$ such that
  $$
   \xymatrix{
             BP   \ar[d]^-{Bpr} \ar[r]^-r & P_{\SAm} C \\
             B\pi \ar[ur]_-{\tilde{r}}    &
             }
  $$
  is a commutative diagram (up to homotopy) and, moreover, $r \simeq \ast$ if and only if $\tilde{r} \simeq \ast$.

  \medskip

  We are left to show that there exists $m_0$ such that for all $m\geq m_0$, $\tilde{r}$ is null-homotopic. Assume first that $\pi \cong \z{/p^n}$. By Lemma \ref{l:BZpm lifts BPinf}, there exists $m_0 \geq 0$ and a
  map $\tilde{f} \colon \bz{p^{m_0}} \rightarrow BP$ such that $Bpr \circ \tilde{f} =  B\alpha$, where $\alpha(1)$ is a generator of $\pi \cong \z{/p^n}$. Moreover, since
  $\bz{p^{m_0}}$ is $\An{m_0}$-cellular there is a map $\bar{f} \colon \bz{p^{m_0}} \rightarrow CW_{\An{m_0}} (BP)$ such that the following diagram
  $$
   \xymatrix{CW_{\An{m_0}} (BP)                                      \ar[r]^-{c} &    BP    \ar[r]^-{r} \ar[d]^-{Bpr} & P_{\SAn{m_0}} C \\
                 \bz{p^{m_0}} \ar[u]^-{\bar{f}} \ar[ur]^-{\tilde{f}} \ar[r]^-{B\alpha} & \bz{p^n} \ar[ur]_-{\tilde{r}}      &
             }
  $$
  is homotopy commutative.  If the homotopy fibration $\bz{p^{m_0-n}}\rightarrow \bz{p^{m_0}}\rightarrow \bz{p^n}$ and the map $r\circ \tilde{f}\colon \bz{p^{m_0}}\rightarrow P_{\SAn{m_0}} C$ satisfy the hypotheses of Zabrodsky's Lemma then $\tilde{r}\simeq *$ if and only if $r\circ \tilde{f}\simeq *$, and this holds since $r\circ \tilde{f}\simeq r\circ c\circ \bar{f}\simeq *$. Let us check the hypotheses. First, $r\circ \tilde{f}$ restricted to $\bz{p^{m_0-n}}$ is null-homotopic since $r\circ \tilde{f}\simeq \bar{r}\circ B\alpha$. Second, $$\map_*(\bz{p^{m_0-n}}, \Omega P_{\SAn{m_0}} C)\simeq \map_*(\bz{p^{m_0-n}}, P_{\An{m_0}} \Omega C)\simeq *$$ since $\bz{p^{m_0-n}}$ is $\bz{p}$-acyclic, hence $\An{m_0}$-acyclic.

  We will proceed to the general case by induction on the order of $\pi$. Let $\{x_1, \ldots, x_n\}$ be a minimal set of generators of $\pi$ and let $H \trianglelefteq \pi$ the normal subgroup generated by $x_1, \ldots, x_{n-1}$ and its conjugates by
  powers of $x_n$. Since $\pi$ is a finite $p$-group, it is nilpotent, and there is a short exact sequence
  $$
   \xymatrix@1{0 \ar[r] & H \ar[r]^-j & \pi \ar[r]^-k & \z{/p^a} \ar[r] & 0},
  $$
  where $\z{/p^a}$ is generated by the image of $x_n$. Taking the pull-back of the projection $Bpr$ along $Bj$, we obtain the diagram
  $$
   \xymatrix{ BP_0 \ar@{=}[r] \ar[d]      & BP_0 \ar[d]                  &              \\
              BP'  \ar[d]^{pr'} \ar[r]^-h & BP \ar[d]^-{Bpr} \ar[r]^-{r} & P_{\SAm} (C) \\
              BH   \ar[r]^-{Bj}           & B\pi \ar[ur]^-{\tilde{r}}    &
             }
  $$
  where $BP'$ is also the classifying space of a discrete $p$-toral group. By induction, there exists $m'_0$ such that, for all $m\geq m_0'$, $BP'$ is $\An{m}$-cellular, hence $r \circ h \simeq \ast$. Therefore $\tilde{r} \circ Bj \simeq \ast$ by Zabrodsky Lemma applied to the left fibration. Note that $BH$ is also $\An{m}$-cellular by Proposition \ref{p:CW fib}.

  Consider now the diagram
  $$
   \xymatrix{ B(\langle x_n \rangle \cap H) \ar[d] \ar[r] & BH \ar[d]^-{Bj} \ar[dr]^-{\tilde{r} \circ Bj \simeq \ast} &            \\
              B(\langle x_n \rangle)        \ar[d] \ar[r] & B\pi \ar[d]     \ar[r]^-{\tilde{r}}                       & P_{\SAm} C. \\
              \bz{p^a}                      \ar@{=}[r]    & \bz{p^a}        \ar@{-->}[ur]^-{r'}                       &
             }
  $$
  We check again that we are in the position of applying Zabrodsky's lemma to the vertical fibrations since, if $G$ is a finite $p$-group, $$\map_*(BG,\Omega P_{\SAm} C)\simeq \map_*(BG, P_{\Am} \Omega C)\simeq *$$ because $BG$ is $\bz{p}$-acyclic (hence $\Am$-acyclic).

  Then there exists a map $r' \colon \bz{p^a} \rightarrow P_{\SAn{m_0}} C$ making the previous diagram commutative and such that $\tilde{r} \simeq \ast$ if and only if
  $r' \simeq \ast$. Or, applied this time to the left fibration, $\tilde{r} |_{B(\langle x_n \rangle)} \simeq \ast$ if and only if $r' \simeq \ast$.
  But now, the same argument use for the situation in which the group of components is cyclic shows that there exists $m_n$ such that for all $m\geq m_n$ with $\tilde{r} |_{B(\langle x_n \rangle)} \simeq \ast$. Now, just let $m_0$ in the theorem to be the maximum between $m_0'$ and $m_n$.
 \end{proof}

\section{Cellularization of classifying spaces of $p$-local compact groups}\label{s:CW clas p loc comp gp}
\label{section4}

In this section we state the main general results about cellularization of $p$-local compact groups. The general pattern will be to present the cellular approximation of the classifying space as the fibre of the rationalization. The role of the Notbohm kernel will be crucial, as this description of the approximation works nicely when this kernel coincides with the $p$-Sylow subgroup of the $p$-local compact group. In this case, the cellularity of $\Linp$ is established up to $p$-completion.

\begin{defi}\label{conditionS}
Let $B$ a connected space. We will say that a space $X$ is $CW_B$-good if $\pi_1(X)$ and $H_2(X;\mathbb Z)$ are finite, and there is a connected $B$-cellular space $B'$ and a map $\varphi \colon B' \rightarrow X$ such that:
\begin{enumerate}[B1.]
\item  $\varphi_{\ast} \colon [B, B']_{\ast} \rightarrow [B,X]_{\ast}$ is surjective,
\item  {\color{black} The abelianization $\pi_1(B')_{ab}$ of $\pi_1(B')$} is a torsion virtually $p$-divisible group,
\item  $\pi_1(\varphi)$ is an epimorphism.
\end{enumerate}
\end{defi}

{\color{black} Observe in particular that B3 always holds when $X$ is 1-connected.}

 \begin{prop}\label{p:CW phi}
  Let $X$ be a $p$-complete $CW_B$-good space. Let $C$ be the homotopy cofibre of $\varphi\colon B'\rightarrow X$ as in Definition \ref{conditionS}.  There is a homotopy fibre sequence  $$\p{CW_B(X)}{p}\stackrel{\p{c}{p}}{\rightarrow} X \stackrel{\p{r}{p}}{\rightarrow} \p{(P_{\Sigma B}(C))}{p}.$$
  The fundamental group  $\pi_1(\p{CW_B(X)}{p})$ is a finite $p$-group. Moreover if $\p{r}{p}$ is null-homotopic, then the $p$-completion $\p{c}{p}$ of the augmentation $c:CW_B(X)\rightarrow X$ is an equivalence.
 \end{prop}

 \begin{proof}
 There is a $B$-cellular space $B'$ and a map $\varphi \colon B' \rightarrow X$ satisfying the properties in Definition \ref{conditionS}.
  The first part of the statement follows from \cite[Theorem 20.3]{MR1408539} since $B'$ is $B$-cellular and $\varphi_{\ast}$ is surjective. Since $C$ is $1$-connected (because $\pi_1(\varphi)$ is surjective), the same holds for $P_{\Sigma B}(C)$ (see the proof of \cite[Lemma 5.3]{CellFusSys} for details) and the $p$-completion of Chacholski's fibration
  $$\p{CW_B(X)}{p}\stackrel{\p{c}{p}}{\rightarrow} X \stackrel{\p{r}{p}}{\rightarrow} \p{(P_{\Sigma B}(C))}{p}$$
  is again a homotopy fibre sequence.

  By \cite[Propositions VII.5.1, VII.6.3]{MR0365573}, since $X$ is $p$-complete, $\pi_1 (X)$ is a finite $p$-group.

  Now, following \cite[Proposition 4.2]{CellFusSys} and using
  Proposition~\ref{p:CW phi}, we get the long exact sequence of homotopy groups
  $$
   \ldots \rightarrow \pi_2 (\p{(P_{\Sigma {B}} C)}{p}) \rightarrow \pi_1 (\p{CW_{B}(X)}{p}) \rightarrow \pi_1 (X) \rightarrow \ldots
  $$
  So we must show that $\pi_2(\p{(P_{\Sigma {B}} C)}{p})$ is a finite $p$-group.

  Notice that
  there exists an epimorphism $H_2(C;\z{}) \twoheadrightarrow H_2(P_{\Sigma {B}} C;\z{}) \cong \pi_2(P_{\Sigma {B}} C)$ by \cite[Proposition 3.2]{CellInfEsp}. Moreover, by
  \cite[Proposition 4.2]{CellFusSys}, the group in the left fits in a short exact sequence
  $$
   0 \rightarrow H_2 \rightarrow H_2(C;\z{}) \rightarrow H_1 \rightarrow 0,
  $$
  where $H_2$ is a quotient of $H_2(X; \z{})$, finite by assumption, and $H_1 \subset H_1(B';\z{}) \cong \pi_1(B')_{ab} \cong (\mathbb Z/p^\infty)^r\times K$ where $K$ is a finite group. Then $H_2 (C; \z{}) \cong (\z{/p^{\infty}})^{n'} \times H'$ where $H'$ is a finite group and hence $\pi_2 (P_{\Sigma {B}} C) \cong (\z{/p^{\infty}})^n \times H$, where
  $H$ is a finite group. Furthermore, $P_{\Sigma {B}} C$ is $1$-connected, and then
  \cite[Proposition VI.5.1]{MR0365573} implies an isomorphism:
  $$
   \pi_2 (\p{(P_{\Sigma B} C)}{p}) \cong \Ext(\z{/p^{\infty}}, \pi_2 (P_{\Sigma {B}} C)) \cong \Ext(\z{/p^{\infty}}, (\z{/p^{\infty}})^n \times H).
  $$
  By \cite[Example VI.4.4 (i)]{MR0365573}, $\Ext(\z{/p^{\infty}},H) \cong \hat{\z{}}_p \otimes H$ a finite $p$-group, and by \cite[Example VI.4.2]{MR0365573},
  $\Ext(\z{/p^{\infty}},\z{/p^{\infty}}) \cong 0$. Therefore {\color{black} $\pi_2 (\p{(P_{\Sigma B} C)}{p})$} is a finite $p$-group. So we are done.

  Now assume that $\p{r}{p}$ is null-homotopic. We obtain a splitting
of the fibre  $\p{CW_{B} (X)}{p} \simeq X \times \Omega \p{(P_{\Sigma {B}} C)}{p}$. Applying the functor $P_{B}$ to the previous equivalence, we get
  $$P_{B}(\p{CW_{B} (X)}{p}) \simeq P_{B}(X) \times \Omega P_{\Sigma {B}}(\p{(P_{\Sigma {B}} C)}{p}).$$

  Now we proceed to $p$-complete the previous equivalence.  One can apply   \cite[Lemma 3.9]{CASTFLOR} and then $\p{P_{B}(\p{CW_{B} (X)}{p})}{p} \simeq \p{P_{B}(CW_{B} (X))}{p} \simeq \ast$. On the right, by the same result we have $\p{P_{\Sigma {B}}(\p{(P_{\Sigma {B}} C)}{p})}{p} \simeq \p{P_{\Sigma {B}}(P_{\Sigma {B}} C)}{p} \simeq \p{(P_{\Sigma {B}} C)}{p}$, since $P_{\Sigma {B}} C$ is $1$-connected, and therefore,
  $\p{(P_{\Sigma {B}} C)}{p} \simeq \ast$ and $\p{c}{p}$ is an equivalence.

 \end{proof}

 \begin{remark}
 In fact, in the proof of Proposition \ref{p:CW phi} we have shown that {\color{black} $\pi_2 (P_{\Sigma B} C)$} is a torsion virtually $p$-divisible group and, therefore, the same holds for $\pi_1(CW_B(X))$.

 The proof of Theorem \ref{p:BS cell CW A Xp = Xp} shows that if $\p{r}{p}$ is null-homotopic for $p$-complete $CW_B$-good spaces then, in fact,  $\p{(P_{\Sigma {B}} C)}{p} \simeq \ast$ and $\p{c}{p}$ is an equivalence.
 \end{remark}

 Next we will show that classifying spaces of $p$-local compact groups satisfy the conditions in Definition \ref{conditionS}, except {\color{black} possibly} for the second homology group being finite.

 \begin{prop} \label{p:Linp+H2}
 Let $(S,\mathcal F)$ be the fusion system of a $p$-local compact group and $\Linp$ its classifying space. If $H_2(\Linp;\mathbb Z)$ is finite, then there is a natural number $n_\mathcal F$ such that $\Linp$ is a $p$-complete $CW_{B_n}$-good space for $B_n=B\mathbb Z/p^\infty \times B\mathbb Z/p^n$ and $n\geq n_{\mathcal F}$.
 \end{prop}

 \begin{proof}
 We have that $\pi_1 (\Linp)$ is a finite $p$-group according to \cite[Theorem B.5]{AlexPaper} and $\Linp$ is $p$-complete (\cite[Proposition 4.4]{MR2302494}).

 Let $[f]\in [B\mathbb Z/p^\infty \times B\mathbb Z/p^n,\Linp]$, {\color{black} for a certain $n\geq 1$ ($n_{\F}$ will be defined a little bit later)}. We consider $f_2\colon B\mathbb Z/p^n\rightarrow \Linp$ the restriction {\color{black} of a representative of $f$} to the second coordinate, and $f_1\colon B\mathbb Z/p^\infty\rightarrow \map(B\mathbb Z/p^n,\Linp)_{f_2}$ the adjoint map of $f$.

 Observe that $\map(B\mathbb Z/p^n,\Linp)_{f_2}$ is again the classifying space of a $p$-local group (see \cite{AlexPaper}) over a discrete $p$-toral group $R$ with a maximal torus (maximal $p$-divisible subgroup $R_0$) of rank $r(f)$.  Then, up to homotopy, the map  $f_1$ factors  $$f_1\colon B\mathbb Z/p^\infty\stackrel{f_1'}{\rightarrow} (B\mathbb Z/p^\infty)^{r(f)}\stackrel{Bi}{\rightarrow} \map(B\mathbb Z/p^n,\Linp)_{f_2}$$ where $(B\mathbb Z/p^\infty)^{r(f)}\rightarrow \map(B\mathbb Z/p^n,\Linp)_{f_2}$ is the inclusion of a maximal torus.

 There is a homotopy commutative diagram,
 $$
   \xymatrix{
   B\mathbb Z/p^{\infty}\times B\mathbb Z/p^n \ar[r]^(0.75){f} \ar[d]_-{f_1} & \Linp \ar@{=}[d] \\
   B(\mathbb Z/p^{\infty})^{r(f)}\times B\mathbb Z/p^n \ar[r]^(0.75){f'} \ar[d]_-{Bi} & \Linp \ar@{=}[d] \\
   \map(B\mathbb Z/p^{n},\Linp)_{f_2}\times B\mathbb Z/p^n \ar[r]^(0.75){ev} & \Linp.
             }
  $$

 Consider $$\varphi\colon  \bigvee_{[f]\in [B\mathbb Z/p^n,\Linp]} (B\mathbb Z/p^\infty)^{r(f)}\times B\mathbb Z/p^n\rightarrow \Linp$$ the evaluation map. We take $B'$ in Definition \ref{conditionS} to be the source space of $\varphi$, {\color{black} and $B=B_n$}.

 The previous remarks imply that Condition B1 in Definition \ref{conditionS} is satisfied. For Condition B2 we need to check that the set $[B\mathbb Z/p^n,\Linp]$ is finite. By \cite[Proposition 6.2]{MR2302494}, there is a surjective map $[B\mathbb Z/p^n,BS]\rightarrow [B\mathbb Z/p^n, \Linp]$ and $[B\mathbb Z/p^n,BS]\cong Rep(\mathbb Z/p^n,S)$ is finite by \cite[Lemma 1.4a]{MR2302494}.

 It remains to check that $\pi_1(\varphi)$ is an epimorphism. By \cite[Proposition 4.4]{MR2302494}, there is an epimorphism $$q\colon S\rightarrow S/T\rightarrow \pi_1(\Linp),$$ being $T$ the maximal torus of $S$. We will show that given $g\in \pi_1(\Linp)$, there is a map $\psi_g\colon B\mathbb Z/p^m\rightarrow \Linp$ such that $g\in \Ima(\pi_1(\psi_g))$. Let $h\in q^{-1}(g)\leq S$ be a preimage of $g$. Then $\langle h \rangle\leq S$ is a finite $p$-group since $S$ is a locally finite $p$-group, that is, $\langle h \rangle\cong \mathbb Z/p^m$ for some $m$ greater or equal  than the order of $g$. The map induced by the inclusion $\psi_g\colon B\mathbb Z/p^m\rightarrow BS\rightarrow \Linp$ induces a homomorphism at the level of fundamental groups which has $g$ in its image.

 Finally, since $\pi_1(\Linp)$ is finite, we can take $n_\mathcal F$ to be the maximum of $m$'s for all $g\in \pi_1(\Linp)$. Then we see that $$\varphi\colon  \bigvee_{[f]\in [B\mathbb Z/p^{n_\mathcal F},\Linp]} (B\mathbb Z/p^\infty)^{r(f)}\times B\mathbb Z/p^{n_\mathcal F}\rightarrow \Linp$$ induces an epimorphism in the fundamental group (note that there are no non-trivial group homomorphisms from $\mathbb Z/p^\infty$ to a finite group).
 \end{proof}

{\color{black} An analogous result is true if we change $B_n$ by $B\mathbb{Z} /p^n$.

 \begin{prop} \label{p:Linp+H2pm}
 Let $(S,\mathcal F)$ be the fusion system of a $p$-local compact group and $\Linp$ its classifying space.
 If $H_2(\Linp;\mathbb Z)$ is finite, then there is a natural number $n_\mathcal F$ such that $\Linp$ is a $p$-complete $CW_{B\mathbb{Z}/p^n}$-good space
 for $n\geq n_{\mathcal F}$.
 \end{prop}

 \begin{proof}

 In the notation of Definition \ref{conditionS}, $B$ will be now $B\mathbb{Z}/p^n$ and $B'$ will be defined as the wedge $\bigvee_{[f]\in [B\mathbb Z/p^n,\Linp]} B\mathbb{Z}/p^n$ of copies of $B\mathbb{Z}/p^n$. We have already seen that $\pi_1 (\Linp)$ is a finite $p$-group. Condition B1 is clear, and as we have seen in the previous proof that the set $[B\mathbb Z/p^n,\Linp]$ is finite, B2 holds by Seifert-van Kampen theorem. The proof of B3 is identical to the one of the previous result, as it is the way of selecting $n_{\F}$.

 \end{proof}}

\begin{remark}
\label{r:1-connected}

Note that in the previous two results, if $\Linp$ is 1-connected, it can be taken $n_{\F}=1$.

\end{remark}

The next proposition and its corresponding corollaries show that the Notbohm kernel is the appropriate adaptation of the strongly closed subgroup $Cl(S)$ of \cite{MR2351607} to the context of $p$-local compact groups.

 \begin{theorem}\label{p:BS cell CW A Xp = Xp}
  Let $(S,\F)$ be a $p$-local compact group such that $H_2(\Linp;\z{})$ is finite. If $\ker (\p{r}{p}) = S$, then the augmentation map
  $c_{\Linp} \colon CW_{B_m}(\Linp) \rightarrow \Linp$ is a mod $p$-equivalence for $m\geq n_{\mathcal F}$ as in Proposition \ref{p:Linp+H2};
  in particular, this is the case if $BS$ is $B_m$-cellular.
 \end{theorem}

 \begin{proof}
  As $B_{\mathcal{F}}$ is $CW_{B_m}$-good by Proposition \ref{p:Linp+H2},  Proposition~\ref{p:CW phi} gives an homotopy fibre sequence
  $$
   \xymatrix@1@C=1.25em{CW_{B_m}(\Linp)^\wedge_p \ar[r]^-{c^\wedge_p} & \Linp \ar[r]^-{r^\wedge_p} & \p{(P_{\Sigma {B_m}} C)}{p}}.
  $$
 Since $\ker(\p{r}{p}) = S$, Proposition~\ref{p:Dwyer p local} implies that $\p{r}{p} \simeq \ast$. By Proposition \ref{p:CW phi} we obtain the result.In particular, if $BS$ is $B_m$-cellular, then the map $BS \rightarrow \Linp$ lifts to $CW_{B_m} (\Linp)$ and hence $\ker(\p{r}{p}) = S$.

 \end{proof}

\begin{coro}\label{c:CW mod p}
Let $(S,\F)$ be a $p$-local compact group such that $H_2(\Linp;\z{})$ is finite.  Then there exists $m_0$ such that for every $m\geq m_0$ the map $c_{\Linp} \colon CW_{B_m}(\Linp) \rightarrow \Linp$ is a mod $p$-equivalence.
\end{coro}

\begin{proof}
We know from Proposition \ref{p:BP Am cellular} that there exists $r_0$ such that, for every $r\geq r_0$, $BS$ is $B_r$-cellular.
Now it is enough to take $m_0=max(n_{\mathcal F}, r_0)$.
\end{proof}

\begin{remark}
{\color{black} The analogous statements hold changing $B_m$ by $B\mathbb{Z}/p^m$ (referred to the respective $\p{r}{p}$)}.
{\color{black} The proof for the case $B\mathbb{Z}/p^m$
invokes Proposition \ref{p:Linp+H2pm} instead of \ref{p:Linp+H2}.}
\end{remark}

 Using these tools, we are also able to understand the structure of the $B_n$-cellular approximation of $\Linp$ \emph{before} $p$-completion.

 \begin{coro}\label{c:BS cell CW A X fib rac}
  Let $(S,\F)$ be a $p$-local compact group such that $H_2(\Linp;\z{})$ is finite and $\Linp$ is nilpotent. Then
  $CW_{B_n}(\Linp)$ fits into a homotopy fibre sequence
  $$
   CW_{B_n}(\Linp) \rightarrow \Linp \rightarrow (\Linp)_{\Q}
  $$ for $n\geq n_{\mathcal F}$.
 \end{coro}

 \begin{proof}
Since the classifying space $\Linp$ is nilpotent, the observation after Theorem \ref{t:chacholski} implies that $CW_{B_n}(\Linp)$ is so, and we can use Sullivan arithmetic square. By \cite[Lemma 2.8]{CASTFLOR}, we have $(CW_{B_n} (\Linp))_{\Q} \simeq \ast$ and $\p{CW_{B_n} (\Linp)}{q} \simeq \ast$ for $q \neq p$, so the statement follows from Theorem \ref{p:BS cell CW A Xp = Xp}.
  \end{proof}

 \begin{remark}
  If $S$ is a finite $p$-group, then $\Linp$ is always nilpotent and Corollary~\ref{c:BS cell CW A X fib rac} holds for any finite $p$-local group $(S,\F)$ (in this case $(\Linp)_{\Q} \simeq \ast$). See \cite[Proposition 4.1]{CellFusSys} for details.
 \end{remark}

  In the proof of Theorem \ref{p:CW phi}, the 1-connectedness of $C$ is essential. In general, if $P$ is a discrete $p$-toral group and we are interested in $BP$-cellularization, the cofibre $C$ might not be $1$-connected. Next result shows that in this case we can compute the cellularization of another $p$-local compact group, whose classifying space is $BP$-equivalent to the first one, and whose Chach\'olski's cofibre is 1-connected.

 \begin{prop}\label{p:new p comp gp}
  Let $(S, \F)$ be a $p$-local compact group, $P$ a discrete $p$-group and $N$ a normal subgroup of $\pi_1 \Linp$ generated by the image of homomorphisms
  $P \rightarrow \pi_1 \Linp$ such that the induced pointed map $BP \rightarrow B\pi_1 \Linp$ lifts to $\Linp$. Then there exists a $p$-local compact group $(S, \F_N)$ and a
  $BP$-equivalence $f \colon \Linp_N \rightarrow \Linp$ such that the homotopy cofibre of $ev \colon \bigvee_{[BP,\Linp_N]_{\ast}} BP \rightarrow \Linp_N$ is
  $1$-connected.
 \end{prop}

 \begin{proof}
  Proceeding as in \cite[Proposition 4.2 (Step 2)]{CellFusSys} we need to check that if $\tilde{\Linp}$ is the $1$-connected cover or $\Linp$ and $X$ is the homotopy pullback of the fibration
  $$ \widetilde{\Linp} \rightarrow \Linp \rightarrow B \pi_1 (\Linp) $$
  along $BN \rightarrow B \pi_1 (\Linp)$, then there exists a $p$-local compact group $(S, \F_N)$ such that $X \simeq \Linp_N$. But this holds by \cite[Theorem B.4.4]{AlexTesis}.
 \end{proof}

\section{Cellularization of classification spaces of compact Lie groups}
\label{section5}

 Let now $G$ be a compact Lie group and $S \in \Syl_p(G)$. The original motivation for our work was to understand the $A$-cellular properties of $BG$ for homotopy meaningful spaces $A$. In \cite{CASTFLOR}, a number of results in this direction were obtained, but we expected there to find a way of interpreting them in an unified way, as well as obtaining new ones; the example of $BSO(3)$ was particularly intriguing, because the structure of $CW_{B\z{/p}}BSO(3)$ strongly resembled the general structure of $CW_{B\z{/p}}BH$ for $H$ finite (getting rid of the rational part), but the method to describe it was defined \emph{ad hoc} and had no apparent relation with the other examples there. The results of this section show that the shape of $CW_{B\z{/2}}BSO(3)$ was the ``standard one", and hence this example worked in fact as a guiding light in our search.

 As can be deduced from the results in the previous section, the key point to describe these cellularizations in the desired unified way is to understand the cellular structure of $BG$ and their homotopic mod $p$ analogues (with regard to classifying spaces of $p$-torsion groups) in terms of the fusion data of the corresponding $p$-local compact group. This should not be surprising, as these data codify the mod $p$ homotopy structure of the corresponding classifying space. In this section, we will make use of these properties in order to obtain the corresponding results for classifying spaces of honest compact Lie groups, and more accurate descriptions of the cellular structure by means of classical classification results (see in particular Theorem \ref{t:CWBGp BGp BGpQ} below). We will concentrate first in $B_m$-cellularization, and then we describe some examples of cellular approximation with regard to $B\mathbb{Z} /p^m$, that will serve as a bridge to the more thorough treatment of last section.


Our first two results translate the conditions of Proposition \ref{p:BS cell CW A Xp = Xp} and Corollary \ref{c:BS cell CW A X fib rac} about the Notbohm kernel to the language of classifying spaces of groups.

 \begin{coro}\label{p:CW BGp = BGp}
  Let $G$ be a compact Lie group and $S \in \Syl_p(G)$ a maximal discrete $p$-toral subgroup. Assume that $H_2(\p{BG}{p}; \z{})$ is finite. Then there is an integer $n_{\mathcal F} \geq 1$ such that
  $\ker(\p{r}{p})$ is isomorphic to the Sylow $S$, and the augmentation map $c \colon CW_{\Am} (\p{BG}{p}) \rightarrow \p{BG}{p}$ is a mod $p$ equivalence.
 \end{coro}

 \begin{proof}
By Proposition \ref{p:BP Am cellular}, $BS$ is $B_m$-cellular for every $m\geq m_0$, and now the result is a straightforward consequence of Theorem~\ref{p:BS cell CW A Xp = Xp}.
\end{proof}


 \begin{coro}\label{c:CWBGp BGp BGpQ}
  Let $G$ be a compact connected Lie group such that $\pi_1 G$ is finite. Then if $\ker(\p{r}{p}) = S$, the homotopy fibre of $\p{BG}{p} \rightarrow (\p{BG}{p})_{\Q}$ is
  homotopy equivalent to $CW_{\Am} (\p{BG}{p})$. Furthermore, there exists an integer $m_0 \geq 1$ such that
  $$
   CW_{\Am} (\p{BG}{p}) \rightarrow \p{BG}{p} \rightarrow (\p{BG}{p})_{\Q}
  $$
  is a fibre sequence for all $m \geq m_0$.
 \end{coro}

 \begin{proof}
  If $G$ is connected, then $\p{BG}{p}$ is $1$-connected, and by Hurewicz Theorem, $H_2(\p{BG}{p};\z{}) \cong \pi_2(\p{BG}{p})$ finite since $\pi_1 G$
  is so. Then the statement follows from Corollary~\ref{c:BS cell CW A X fib rac}.

 \end{proof}

 Next we will remove the hypothesis of finiteness of $\pi_1 G$ of the classification of compact connected Lie groups. The following example illustrates that frequently this assumption is unnecessary.

 \begin{example}\label{e:S1}
  The compact connected Lie group $S^1$ has an infinite cyclic fundamental group, and moreover $\p{(BS^1)}{p} \simeq \p{(\bz{p^{\infty}})}{p} \simeq K(\hat{\z{}}_p,2)$.
  Using Sullivan Arithmetic Square we get the fibre sequence
  $$
   \xymatrix@1@C=1.5em{K(\z{/p^{\infty}},1) \ar[r]^-{\iota} & K(\hat{\z{}}_p,2) \ar[r] & K(\hat{\z{}}_p,2)_{\Q}}.
  $$
  Since the base space is $\Am$-null, the map $\iota$ is a $\Am$-equivalence. The fibre $\bz{p^{\infty}} \simeq K(\z{/p^{\infty}},1)$ is $\Am$-cellular. Then the fibre sequence
  becomes
  $$
   CW_{\Am} (\p{(BS^1)}{p}) \rightarrow \p{(BS^1)}{p} \rightarrow (\p{(BS^1)}{p})_{\Q}.
  $$
 \end{example}

The following describes the general situation in the compact connected case. Observed that the hypothesis about the finiteness of the second homology group is no further necessary.

 \begin{theorem}\label{t:CWBGp BGp BGpQ}
  Let $G$ be a compact connected Lie group. Then there exists an integer $m_0 \geq 1$ such that the homotopy fibre of $\p{BG}{p} \rightarrow (\p{BG}{p})_{\Q}$ is homotopic to
  $CW_{\Am} (\p{BG}{p})$ for all $m \geq m_0$.
 \end{theorem}

 \begin{proof}
  A compact connected Lie group $G$ is homeomorphic to $G \cong H/K$, where $H = G_1 \times \ldots \times G_k \times T^r$, $G_i$ is a $1$-connected simple Lie group for all
  $i \in \{1, \ldots, k\}$, $T^r = (S^1)^r$ and $K$ is a finite subgroup of the center of $H$ (see \cite[Theorem V.8.1]{MR781344}). Therefore we have a central extension
  $$
   0 \rightarrow K \rightarrow H \rightarrow G \rightarrow 0,
  $$
  which induces a principal fibration.
  Hence by \cite[Lemma II.5.1]{MR0365573}, we can $p$-complete to get the homotopy fibre sequence
  $$
   \p{BK}{p} \rightarrow \p{BH}{p} \rightarrow \p{BG}{p}.
  $$

 Every $G_i$ is 1-connected, and then Corollary~\ref{c:CWBGp BGp BGpQ} implies that for all $i \in \{1, \ldots, k\}$ there exists an integer $m_i \geq 0$ and a homotopy fibration
  $$
   CW_{\Amm{i}} (\p{(BG_i)}{p}) \rightarrow \p{(BG_i)}{p} \rightarrow (\p{(BG_i)}{p})_{\Q}.
  $$
 Moreover by Example~\ref{e:S1}, we obtain another homotopy fibration
  $$
   CW_{\An{n}} (\p{(B(S^1)^r)}{p}) \rightarrow \p{(B(S^1)^r)}{p} \rightarrow (\p{(B(S^1)^r)}{p})_{\Q}
  $$
  for all $n \geq 0$.

Let $l_1 = \max\{m_1, \ldots, m_k\} \geq 0$. Observe that $\p{BK}{p}$ is the classifying space of an abelian finite $p$-group, and hence $\p{BK}{p}$ is $\bz{p^{l_2}}$-cellular for a certain non-negative integer $l_2 \geq 0$. If $m = \max\{l_1,l_2\}$, then $\p{BK}{p}$ is $\Am$-cellular and we have the fibration
  $$
   CW_{\Am} (\p{BH}{p}) \rightarrow \p{BH}{p} \rightarrow (\p{BH}{p})_{\Q}.
  $$

  Now, since $BG$ is $1$-connected, $\p{BG}{p}$ is so and  \cite[Lemma II.5.1]{MR0365573} gives us the fibre sequence
  $$
   (\p{BK}{p})_{\Q} \rightarrow (\p{BH}{p})_{\Q} \rightarrow (\p{BG}{p})_{\Q},
  $$
  where $(\p{BK}{p})_{\Q} \simeq \ast$ since $\p{BK}{p}$ is the classifying space of a finite group. So there is an equivalence $(\p{BH}{p})_{\Q} \simeq (\p{BG}{p})_{\Q}$.

  Let $F$ be the homotopy fibre of $\p{BG}{p} \rightarrow (\p{BG}{p})_{\Q}$. We have the following homotopy commutative diagram of fibrations
  $$
   \xymatrix{\p{BK}{p} \ar[r] \ar@{=}[d] & CW_{\Am}(\p{BH}{p}) \ar[r] \ar[d]           & F \ar[d]^-{f} \\
             \p{BK}{p} \ar[r] \ar[d]     &          \p{BH}{p}  \ar[r] \ar[d]           & \p{BG}{p} \ar[d] \\
               \ast    \ar[r]            &         (\p{BH}{p})_{\Q}   \ar[r]^-{\simeq} & (\p{BG}{p})_{\Q}
             }
  $$
  The right vertical fibration induces a homotopy fibration in mapping spaces
  $$
   \xymatrix@1@C=1.5em{\mapp(\Am,F) \ar[r]^-{f_{\ast}} & \mapp(\Am,\p{BG}{p})_{\{c\}} \ar[r] & \mapp(\Am, (\p{BG}{p})_{\Q})_c}
  $$
  where $\mapp(\Am, (\p{BG}{p})_{\Q}) \simeq \ast$ because $(\p{BG}{p})_{\Q} \simeq L_{\Q} (\p{BG}{p})$ ($\p{BG}{p}$ is $1$-connected), and $\tilde{H}_{\ast}(\Am; \Q) \cong 0$.
  Therefore, $f$ is a $\Am$-equivalence and then $CW_{\Am} F \simeq CW_{\Am} (\p{BG}{p})$.

  Finally, since $\p{BK}{p}$ and $CW_{\Am}(\p{BH}{p})$ are $\Am$-cellular spaces, $F$ is also $\Am$-cellular according to \cite[Theorem 2.D.11]{MR1392221}.
 \end{proof}

 \begin{coro}\label{c:BG K-cell}
  For any compact connected Lie group $G$ there exists an integer $m_0 \geq 1$ such that $\p{BG}{p}$ is $K(\Q \times \z{/p^{\infty}} \times \z{/p^m},1)$-cellular for all
  $m \geq m_0$.
 \end{coro}

 \begin{proof}
  Let $G$ be a compact connected Lie group. According to Theorem~\ref{t:CWBGp BGp BGpQ}, there is an integer $m_0 \geq 1$ such that for all $m \geq m_0$ we have the homotopy
  fibration
  $$
   \Omega ((\p{BG}{p})_{\Q}) \rightarrow CW_{\Am}(\p{BG}{p}) \rightarrow \p{BG}{p}.
  $$

 Note that $\Omega ((\p{BG}{p})_{\Q}) \simeq (\p{(\Omega BG)}{p})_{\Q} \simeq (\p{G}{p})_{\Q}$ since $BG$ is $1$-connected.
  Moreover, there is a homotopy equivalence $(\p{G}{p})_{\Q} \simeq (\prod_{i=1}^n \p{(S^{k_i})}{p})_{\Q} \simeq \prod_{i=1}^n K(\Qp, k_i)$, where $k_i$ are odd natural numbers (see \cite[Section 19]{MR658304}). Each Eilenberg-MacLane space $K(\Qp, k_i)$ is $K(\Qp, 1)$-cellular by \cite[Proposition 3.C.8]{MR1392221}. Since $\Qp$ is an infinite $\Q$-vector space, $K(\Qp, 1)$ is
  $K(\Q,1)$-cellular. This implies that $(\p{G}{p})_{\Q}$ is $K(\Q,1)$-cellular.

  Finally, both $(\p{G}{p})_{\Q}$ and $CW_{\Am}(\p{BG}{p})$ are $K(\Q \times \z{/p^{\infty}} \times \z{/p^m},1)$-cellular, then $\p{BG}{p}$ is so by \cite[2.D]{MR1392221}.
 \end{proof}

It is a very interesting question to understand the nature of the spaces that lie in the cellular class generated by the spaces $K(\Q \times \z{/p^{\infty}} \times \z{/p^m},1)$, $m\geq 1$, and in particular to understand if any ``classifying-like" spaces lies in some of these classes, up to mod $p$ homology equivalence. This problem is related with the theory of $v_0$-spaces, see \cite{MR2242617}.

 In the rest of this section, concrete examples of cellular approximations for $\p{BG}{p}$ are explicitly described by means of the previous theory.

 \begin{example}
  Let $G = S^1$. It is stated in Example~\ref{e:S1} that for every $m \geq 0$, there is a fibre sequence
  $$
   CW_{\Am} (\p{(BS^1)}{p}) \rightarrow \p{(BS^1)}{p} \rightarrow (\p{(BS^1)}{p})_{\Q}.
  $$

  Now we want to compute now the cellularization with respect to $\bz{p^m}$ for a fixed $m \geq 1$. As $\bz{p^m} \rightarrow \bz{p^{\infty}}$ is a $\bz{p^m}$-equivalence, then
  $CW_{\bz{p^m}}(\bz{p^{\infty}}) \simeq \bz{p^m}$. Using the same argument as in Example~\ref{e:S1}, it is easy to see that
  $CW_{\bz{p^m}}(\p{(BS^1)}{p}) \simeq CW_{\bz{p^m}}(\bz{p^{\infty}}) \simeq \bz{p^m}$.
 \end{example}

 \begin{example}\label{e:S3}
  Let $G=S^3$ and $p=2$. Then $S = Q_{2^{\infty}}$, where $Q_{2^{\infty}} = \varinjlim_n Q_{2^{n+1}}$ and $Q_{2^k}$ denotes the generalized quaternion group given by
  $Q_{2^k} = ( \z{/2^{k-1}} \rtimes \z{/4} ) / \langle (2^{k-2},2) \rangle$. For $n \geq 2$, $Q_{2^{n+1}} = \langle x,y \rangle$ such that
  \begin{enumerate}[(i)]
   \item $x^{2^n} = y^4 = 1$,
   \item if $g \in Q_{2^{n+1}}$, then $g = x^a$ or $g = x^a y$ for certain $a \in \z{}$,
   \item $x^{2^{n-1}} = y^2$,
   \item for all $g \in Q_{2^{n+1}}$ such that $g \notin \langle x \rangle$ we have $gxg^{-1} = x^{-1}$, and in particular $yxy^{-1} = x^{-1}$.
  \end{enumerate}

  Denote by $o(a)$ the order of the element $a$. Note first that since $x = yx^{-1} y^{-1}$, $Q_{2^{n+1}} = \langle y, yx^{-1} \rangle$. Moreover, $o(y) = o(yx^{-1}) = 4$ as $yx^{-1}yx^{-1} = yyxx^{-1} = y^2$. Hence
  $BQ_{2^{n+1}}$ is $\bz{4}$-cellular for all $n \geq 2$ according to \cite[Corollary 2.5]{MR2262843}.

  For $n=1$, $Q_4 = \langle x, y \rangle$, where $o(x) = o(y) = 2$, so $BQ_4$ is $\bz{2}$-cellular and, in particular, $\bz{4}$-cellular.

  Therefore, $BQ_{2^{n+1}}$ is $\bz{4}$-cellular for all $n \geq 1$ and hence $BS = BQ_{2^{\infty}}$ is $\bz{4}$-cellular since it is a pointed homotopy colimit of $\bz{4}$-cellular
  spaces. Then, by Theorem~\ref{t:CWBGp BGp BGpQ} there is a homotopy fibration
  $$
   CW_{\bz{2^m}}(\p{(BS^3)}{2}) \rightarrow \p{(BS^3)}{2} \rightarrow (\p{(BS^3)}{2})_{\Q},
  $$
  for all $m \geq 2$.

  The case $m = 1$ is solved by a different method in \cite[Example 6.10]{CASTFLOR}, where it is obtained that $CW_{\bz{2}}(\p{(BS^3)}{2}) \simeq \bz{2}$. The basic reason for this equivalence is that $S^3$ has a special property: there is  only one conjugacy class of elements of order $2$ which is central.

 \end{example}

 \begin{example}\label{ex:so}
  Let $G = SO(3)$ and $p = 2$. In this case $S = D_{2^{\infty}}$, where $D_{2^{\infty}} = \varinjlim_n D_{2^n}$, the colimit of the dihedral groups. The space $BD_{2^{\infty}}$ is
  $\bz{2}$-cellular (see \cite[Example 5.1]{MR2272149}). Then, Theorem~\ref{t:CWBGp BGp BGpQ} gives the fibre sequence
  $$
   CW_{\bz{2^m}}(\p{BSO(3)}{2}) \rightarrow \p{BSO(3)}{2} \rightarrow (\p{BSO(3)}{2})_{\Q},
  $$
  for all $m \geq 1$. As we stated above, the case $m = 1$ was also proved in \cite[Proposition 6.17]{CASTFLOR} using a different method.
 \end{example}

Our last example generalizes the situation of \cite[Section 5]{MR2351607}, where the role of the fusion in the cellular structure of a classifying space was identified for the first time. In Example \ref{ex_son} below, we describe the $\bz{2^m}$-cellular approximation of $BSO(n)$ using a different approach based in the relation of $SO(n)$ and $Spin(n)$.

 \begin{example}\label{e:NGS controls fus}
  Let $G$ be a compact connected Lie group, $S \in \Syl_p(G)$ and $\F := \F_S(G)$. Let $p$ be a prime number such that $(p, |W_G|) = 1$. Then, $N_G(S)$ controls fusion in $G$. Hence,
  proceeding as in \cite[Example 6.4]{CellFusSys}, we get the homotopy fibration
  $$
   CW_{\bz{p^m}}(\p{BG}{p}) \rightarrow \p{BN_G(S)}{p} \rightarrow \p{B(N_G(S)/Cl_{\F}(\z{/p^m}))}{p},
  $$
  for all $m \geq 1$.
 \end{example}

\section{$B\mathbb Z/p^m$-cellularization of classifying spaces}
\label{section6}


Our aim in this last section is to understand the structure of $\Linp$ (and of $BG^{\wedge}_p$, in particular) that depends \emph{only} on $\bz{p^m}$. This was in fact the original motivation for \cite{CASTFLOR}: a direct extension of the previous work for classifying spaces of finite groups. Of course it is the natural framework of \cite{CellFusSys}, since classifying spaces of $p$-local finite groups are rationally trivial.

 For $p$-local compact groups, recall that if $\ker(\p{r}{p})=S$ we are in the situation of Theorem \ref{p:BS cell CW A Xp = Xp}. We will show here that if $\ker(\p{r}{p})$ is strictly contained in the $p$-Sylow subgroup, it will be possible to compute the cellular approximation with respect to $\bz{p^m}$ if there is a factorization of $\p{r}{p}$, up to homotopy,
 through a map $\p{\tilde{r}}{p} \colon \Linp' \rightarrow \p{(P_{\sbz{p^m}}C)}{p}$ with trivial kernel. Here $\Linp'$ the classifying space of another $p$-local compact group.

 \begin{prop}[{\cite[Proposition~6.2]{CellFusSys}}]\label{p:factor}
  Let $(S,\F)$ be a $p$-local compact group such that
  $$
   \p{CW_{\bz{p^m}}(\Linp)}{p} \rightarrow \Linp \rightarrow \p{(P_{\sbz{p^m}} C)}{p}
  $$
  is a fibre sequence. Assume that there is a $p$-local compact group $(S', \F')$ and a factorization (up to homotopy)
  $$
   \xymatrix@C=1.5em@R=1.5em{ \Linp \ar@{=}[r] \ar[d]_-{\pi}     & \Linp \ar[d]^-{\p{r}{p}} \\
                             \Linp' \ar[r]^-{\p{\tilde{r}}{p}} & \p{(P_{\Sigma \bz{p^m}} C)}{p}
                             }
  $$
  such that $\ker(\p{\tilde{r}}{p}) = \{e\}$ and $F$, the homotopy fibre of $\pi \colon \Linp \rightarrow \Linp'$, is $\bz{p}$-acyclic.
  Then $\p{\tilde{r}}{p}$ is a homotopy equivalence and in particular, $\p{CW_{\bz{p^m}}(\Linp)}{p} \simeq F$.
 \end{prop}

 \begin{proof}
  The same proof of \cite[Proposition~6.2]{CellFusSys} holds for $p$-local compact groups.
 \end{proof}

 In the case of compact connected Lie groups, the following theorem of D. Notbohm (combined with Proposition~\ref{p:factor}) allows to obtain a stronger result. From now on, and following his
 convention, $\ker(f)$ will be the closure of $\{g \in S \mid f|_{B\langle g \rangle} \simeq \ast \}$.

 \begin{theorem}[{\cite[Theorem 1.5]{MR1286829}}]\label{t:factor BG BH}
  Let $G$ be a compact connected Lie group and assume that there is a pointed map $f \colon \p{BG}{p} \rightarrow Z$, where $Z$ is a $p$-complete and $\sbz{p}$-null space. Then,
  there exists a compact Lie group $H$ and a commutative diagram
  $$
   \xymatrix{\p{BG}{p} \ar@{=}[r] \ar[d]_-{\pi} & \p{BG}{p} \ar[d]^-f \\
             \p{BH}{p} \ar[r]^-{\bar{f}}        & Z
             }
  $$
  such that $\ker (\bar{f}) = \{e\}$. Moreover, the homotopy fibre of $\pi$ is equivalent to $\p{B\Gamma}{p}$, where $\Gamma$ is a compact Lie group.
 \end{theorem}

 Let us briefly describe the construction of the group $H$ and the map $\pi$, which is a consequence of the classification of compact connected Lie groups. Recall first that  if $G$ is a compact connected Lie group, there is a central extension of compact Lie groups
 $$
  \xymatrix@1@C=1.5em{1 \ar[r] & K \ar[r] & \tilde{G} \ar[r]^-{\alpha} & G \ar[r] & 1},
 $$
 where $\tilde{G} = G_1 \times \ldots \times G_k \times T$, $G_i$ is a $1$-connected simple Lie group for all $i \in \{1, \ldots, k\}$ and $K$ is a finite central subgroup of
 $\tilde{G}$. For each  1-connected Lie group $M$ and a prime $p$, we define a subgroup $H(M,p)\leq M$, which contains the maximal discrete $p$-toral, as follows:
 $$
  H(M,p) = \left\{
            \begin{array}{ll}
              N_M(T), & \mbox{ if $(p,|W_M|) = 1$,} \\
              SU(2) \rtimes \z{/2}, & \mbox{ if $M = G_2$ and $p = 3$,} \\
              M, & \mbox{ otherwise.}
            \end{array}
           \right.
 $$
 In each case, the inclusion $H(M,p) \hookrightarrow M$ induces a mod $p$ equivalence on classifying spaces. If $G_s = G_1 \times \ldots \times G_k$ we denote $H(G_s,p) = H(G_1,p) \times \ldots \times H(G_k,p)$.

 Let $T\leq G_s$ be a strongly closed subgroup in a maximal $p$-toral discrete subgroup. D.Notbohm \cite[Proposition 4.3]{MR1286829} shows that one can split $G_s \cong G' \times G''$ in such a way that
 $T \cong S' \times \tilde{\Gamma}$, where $S' \in \Syl_p(G')$ and $\tilde{\Gamma} \subset T_{G''} \times T$ is a normal subgroup of $H(G'',p) \times T$ (see \cite[Proposition 4.3]{MR1286829}). Then, the inclusion $\iota \colon (G' \times H(G'',p) \times T)/K \hookrightarrow G$ induces a mod $p$ equivalence on classifying spaces.

 Notbohm defines
 $H = (H(G'',p) \times T)/\tilde{\Gamma}$. There is a quotient homomorphism $q : (G' \times H(G'',p) \times T)/K \rightarrow (H(G'',p) \times T)/K \rightarrow (H(G'',p) \times T)/\tilde{\Gamma} = H$ (since the
 kernel of the projection $G' \times H(G'',p) \times T \rightarrow H$ is $G' \times \tilde{\Gamma}$, which contains $K$). Therefore,
 it is defined $\pi = \p{Bq}{p} \circ (\p{B\iota}{p})^{-1}$. The homotopy fibre of $\pi$ is $\p{B((G' \times \tilde{\Gamma})/K)}{p}$, so {\color{black} he takes} $\Gamma =( G' \times \tilde{\Gamma})/K$. See \cite[Section 4]{MR1286829} for more details of this construction.

 \begin{defi}\label{def:normalclosure}
 Let $G$ be a compact connected Lie group and $T$ a strongly $\mathcal F$-closed subgroup. We say that ${\color{black} \Gamma\leq G}$ as constructed is the \emph{normal closure} of $T$ in $G$ at a prime $p$.
 \end{defi}

In our situation, if $T$ is the kernel of $\p{r}{p}\colon \p{BG}{p}\rightarrow \p{(P_{\Sigma B\mathbb Z/p^m}(C))}{p}$, we obtain that this map factors through $\pi\colon \p{BG}{p}\rightarrow \p{BH}{p}$ with trivial kernel, and homotopy fibre $\p{B\Gamma}{p}$.

This description allows to understand the $\bz{p^m}$-cellular structure for the case of compact connected Lie groups.

 \begin{theorem}\label{t:CW Bzpm BG}
  Let $m \geq 1$, $G$ be a compact connected Lie group. Then $CW_{\bz{p^m}}(\p{BG}{p})$ is homotopy equivalent to the homotopy fibre of the rationalization
  $\p{B\Gamma}{p} \rightarrow (\p{B\Gamma}{p})_{\Q}$, where $\Gamma$ is the normal closure {\color{black} at the prime $p$} of the strongly closed subgroup $\ker(\p{r}{p})$.
 \end{theorem}

 \begin{proof}
  From Theorem~\ref{t:factor BG BH} and Proposition~\ref{p:factor} we get that $\p{CW_{\bz{p^m}}(\p{BG}{p})}{p} \simeq \p{B\Gamma}{p}$. Since $\p{BG}{p}$ is 1-connected, we can
  apply Sullivan Arithmetic Square to $CW_{\bz{p^m}}(\p{BG}{p})$ and obtain the desired result since, by \cite[Lemma 2.8]{CASTFLOR}, $CW_{\bz{p^m}}(\p{BG}{p})$ is $H\Q$-acyclic and
  $H\mathbb{F}_q$-acyclic for all $q \neq p$.
 \end{proof}

 {\color{black}
 The next consequence of this result will be very useful:

 \begin{coro}
 \label{c:GammaG}

 Let $G$ be a compact connected Lie group, $m\geq 1$. If $Cl_{\mathcal{F}(G)}(\mathbb{Z}/p^m)$ is a $p$-Sylow in $G$, then $CW_{\bz{p^m}}(\p{BG}{p})$ is homotopy equivalent to the homotopy fibre of the rationalization $\p{BG}{p} \rightarrow (\p{BG}{p})_{\Q}$.

 \end{coro}

\begin{proof}

The condition about $Cl_{\mathcal{F}(G)}(\mathbb{Z}/p^m)$ implies that $\Gamma=G$. Now we apply Theorem \ref{t:CW Bzpm BG}.

\end{proof}}
 If the group $G$ is in addition 1-connected and simple, we can give a more accurate description for $CW_{\bz{p^m}}(BG^{\wedge}_p)$ using as input the classification of their
 strongly closed subgroups given in \cite[Proposition 4.3]{MR1286829}. It is likely that the generalization of this description to $p$-local compact groups will depend on a classification
 of the strongly closed subgroup of these objects, which is unknown at this point.

 Basically, D. Notbohm states that if $G$ is a $1$-connected compact simple Lie group and $K \leq S \in \Syl_p(G)$ is a strongly closed subgroup in $G$, then $K = S$ or $K$ is a
 finite $p$-group. Moreover, if $K$ is a finite group one of the following situations will hold:
 \begin{enumerate}[(a)]
  \item If $(p,|W_G|) = 1$, then $K$ is central in $N_G(S)$.
  \item If $(p,|W_G|) \neq 1$, then:
  \begin{enumerate}[(i)]
   \item  If $G \neq G_2$ or $p \neq 3$, then $K$ is central in $G$.
   \item If $G = G_2$ and $p = 3$, then $K$ is central in $SU(3) \leq G_2$.
  \end{enumerate}
 \end{enumerate}

 Note that situation (a) is described in Example \ref{e:NGS controls fus} and the case $G = S^3$ and $p = 2$ in Example \ref{e:S3}. Now we are in position of stating our description of $CW_{\bz{p^m}}\p{BG}{p}$.

 \begin{prop}\label{p:G simple}
  Let $G$ be a compact $1$-connected simple Lie group. Let $p$ be a prime such that $p \mid |W_G|$. Then for all $m \geq 1$, the $\bz{p^m}$-cellularization of $\p{BG}{p}$ is
  equivalent to the homotopy fibre of the rationalization $\p{BG}{p} \rightarrow (\p{BG}{p})_{\Q}$, unless $G = S^3$ and $p^m = 2$.
 \end{prop}

 \begin{proof}
  The result is proved in case $G = S^3$ and $p = 2$ for $m \geq 2$ in Example~\ref{e:S3}. Then we will consider $G \neq S^3$ or $p \neq 2$.

  Fix $m \geq 1$ and $S \in \Syl_p(G)$, and let $\F := \F_S(G)$. The subgroup $Cl_{\F}(\z{/p^m}) \leq S$ is strongly $\F$-closed by definition and, moreover,
  $Cl_{\F}(\z{/p^m}) \leq \ker(\p{r}{p})$. According to \cite[Proposition~4.3]{MR1286829}, $Cl_{\F}(\z{/p^m}) = S$ or $Cl_{\F}(\z{/p^m})$ is a finite $p$-group; we will see that the second option will never be possible.  Assume by contradiction that $Cl_{\F}(\z{/p^m})$ is a finite $p$-group. Then
  \begin{enumerate}[(i)]
   \item if $G \neq G_2$ or $p \neq 3$, then $Cl_{\F}(\z{/p^m})$ should be central in $G$. According to the classification,  $Z(G) \cong \z{/n}$ for some $n \geq 1$ or $\z{/2} \times \z{/2}$,  depending on the group
         $G$, but there are always non-central elements of order $p$ except when $G= S^3$ and $p = 2$. Therefore, $Cl_{\F}(\z{/p^m})$ cannot be finite and
         hence $Cl_{\F}(\z{/p^m}) = S$.
   \item If $G = G_2$ and $p = 3$, then $Cl_{\F}(\z{/3^m})$ should be central in $SU(3) \leq G_2$, where $Z(SU(3)) = \z{/3}$. Therefore, $Cl_{\F}(\z{/3^m})$ should be $\z{/3}$, but again
         there exists non-central elements of order $3$, and hence $Cl_{\F}(\z{/3^m}) = S$.
  \end{enumerate}


  Thus, in every case, $Cl_{\F}(\z{/p^m}) = S$, so by Corollary~\ref{c:GammaG} we get the homotopy fibration
  $$
   CW_{\bz{p^m}}(\p{BG}{p}) \rightarrow \p{BG}{p} \rightarrow (\p{BG}{p})_{\Q}.
  $$
 \end{proof}

 Theorem~6.9 in \cite{CASTFLOR} tells us that the $\bz{p}$-cellularization of $\p{BG}{p}$, where $G$ is a compact connected Lie group, is the classifying space of a $p$-group
 generated by order $p$ elements, or else it has an infinite number of non-trivial homotopy groups. Note that Proposition~\ref{p:G simple} gives us a more specific result in the
 case of a compact $1$-connected simple Lie group, that is, if the $\bz{p^m}$-cellularization of $\p{BG}{p}$ has infinite non-trivial homotopy groups, then it is equivalent to the
 homotopy fibre of the rationalization $\p{BG}{p} \rightarrow (\p{BG}{p})_{\Q}$. This also remarks in particular the singularity of the case $G=S^3$ and $p=2$, as it is the only of
 these groups where the $B\z{/p^m}$-cellular approximation produces an Eilenberg MacLane $K(G,1)$, with $G$ a finite $p$-group, as an output (see \cite[Example 6.10]{CASTFLOR}. This is owed to the
 existence of just one non-trivial pointed map (up to homotopy) from $B\z{/2}$ to $BS^3$, which is central.

 Next we will describe the $\bz{2^m}$-cellularization of $\p{BSO(n)}{2}$ for every positive $m$. This result generalizes \cite[Proposition 6.17]{CASTFLOR} and Example \ref{ex:so}, and allows to integrate $CW_{\bz{2^m}}BSO(3)$ in the general picture of approximations of classifying spaces of Lie groups. This was one of the remaining questions of the referred previous work of the first two authors, who characterized there $CW_{\bz{2}}BSO(3)$ by completely different tools that only depended on the particular subgroup structure of SO(3).

 \begin{example}
 \label{ex_son}
  Let $n \geq 2$. Consider the fibre sequence $\bz{2} \rightarrow \p{BSpin(n)}{2} \rightarrow \p{BSO(n)}{2}$ associated to the universal cover $Spin(n)\rightarrow SO(n)$. Since $\p{BSpin(n)}{2}$ is 1-connected, it is also $\Q$-good, and  \cite[Lemma II.5.1]{MR0365573} implies $H_{\ast}(\p{BSO(n)}{2}; \Q) \cong H_{\ast}(\p{BSpin(n)}{2}; \Q)$ because $BSO(n)$ is also 1-connected and $\tilde{H}_{\ast}(\bz{2}; \Q) \cong 0$. Then, according to Proposition~\ref{p:G simple}, there is a homotopy
  commutative diagram of fibrations
  $$
   \xymatrix@C=1.5em@R=1.5em{
                            \bz{2}         \ar@{=}[r] \ar[d] & \bz{2}          \ar[r] \ar[d] & \ast \ar[d] \\
            CW_{\bz{2^m}}(\p{BSpin(n)}{2}) \ar[r] \ar[d]     & \p{BSpin(n)}{2} \ar[r] \ar[d] & (\p{BSpin(n)}{2})_{\Q} \ar[d] \\
                             F             \ar[r]^-f         & \p{BSO(n)}{2}   \ar[r]        & (\p{BSO(n)}{2})_{\Q}
                             }
  $$
  Since $CW_{\bz{2^m}}(\p{BSpin(n)}{2})$ and $\bz{2}$ are $\bz{2^m}$-cellular, $F$ is so by Proposition \ref{p:CW fib}; moreover,  $(\p{BSO(n)}{2})_{\Q}$ is a $\bz{2^m}$-null space, and hence $f$ is a $\bz{2^m}$-equivalence. Then we obtain the following fibration
  $$
   CW_{\bz{2^m}}(\p{BSO(n)}{2}) \rightarrow \p{BSO(n)}{2} \rightarrow (\p{BSO(n)}{2})_{\Q}.
  $$ which identifies $CW_{\bz{2^m}}(\p{BSO(n)}{2})$.

\end{example}

With a little bit of additional work, we can compute the correspondent cellularizations of the classifying spaces of the orthogonal groups, which are not connected.

\begin{example}

  Consider now the inclusion $SO(n)\hookrightarrow O(n)$, which gives rise to the fibration  $\p{BSO(n)}{2} \rightarrow \p{BO(n)}{2} \rightarrow \bz{2}$ that has in turn a homotopy section $s \colon \bz{2} \rightarrow \p{BO(n)}{2}$. Using
  the Serre spectral sequence one can check that there is an equivalence $H_{\ast}(\p{BSO(n)}{2}; \Q) \cong H_{\ast}(\p{BO(n)}{2}; \Q)$. Hence, if $F'$ denotes the homotopy fibre of
  $\p{BO(n)}{2} \rightarrow (\p{BO(n)}{2})_{\Q}$, then we get the homotopy commutative diagram
  $$
   \xymatrix@C=1.5em@R=1.5em{
            CW_{\bz{2^m}}(\p{BSO(n)}{2}) \ar[r] \ar[d]       & \p{BSO(n)}{2} \ar[r] \ar[d]    & (\p{BSO(n)}{2})_{\Q} \ar@{=}[d] \\
                             F'          \ar[r]^-g \ar[d]    & \p{BO(n)}{2}  \ar[r]^-h \ar[d] & (\p{BO(n)}{2})_{\Q}  \ar[d]     \\
                \bz{2} \ar@{=}[r] \ar@/^1em/[u]^-{\tilde{s}} & \bz{2} \ar[r] \ar@/^1em/[u]^-s & \ast
                             }
  $$
  where the arrows are homotopy fibrations and the section $\tilde{s}$ of the left vertical fibration exists because the composite $h \circ s$ is null-homotopic. Since
  $(\p{BO(n)}{2})_{\Q}$ is $\bz{2^m}$-acyclic, the map $g$ is a $\bz{2^m}$-equivalence. Furthermore, as  $CW_{\bz{2^m}}(\p{BSO(n)}{2})$ and $\bz{2}$ are $\bz{2^m}$-cellular and the left vertical fibration has a section, \cite[Theorem~4.7]{MR1408539} implies that $F'$ is  $\bz{2^m}$-cellular too. Therefore,
  $$
   CW_{\bz{2^m}}(\p{BO(n)}{2}) \rightarrow \p{BO(n)}{2} \rightarrow (\p{BO(n)}{2})_{\Q}
  $$
  is again a fibration for all $m \geq 1$.
 \end{example}

\color{black}{
We finish by giving a method that in in some cases allows to compute explicitly the cellular approximation of $p$-local compact groups that do $not$ arise from the fusion of honest Lie groups. We start with an easy lemma.

\begin{lemma}
\label{l:inclusion}

Let $(S,\F)$ and $(S,\F')$ $p$-local compact groups over the same $p$-Sylow subgroup, with $\F'<\F$. Assume that for a certain $m\geq 1$ we have $Cl_{\F'}(\bz{p^m})=S$. Then $Cl_{\F}(\bz{p^m})=S$.

\end{lemma}

\begin{proof}

As $\F'<\F$, $Cl_{\F'}(\bz{p^m})\leq Cl_{\F}(\bz{p^m})$. Then $S=Cl_{\F'}(\bz{p^m})\leq Cl_{\F}(\bz{p^m})\leq S$ and the statement follows.

\end{proof}

Our key result in this context will be the following:

\begin{prop}
\label{p:exotic}

Let $(S, \F(G))$ be the $p$-local compact group associated to a 1-connected simple compact Lie group $G$ such that $G\neq S^3$ if $p=2$.
Assume $p \mid |W_G|$. Let $(S,\F)$ another $p$-local compact group over the same $p$-Sylow such that $\F(G)\leq \F$, $B\F$ is 1-connected and $H_2(B\F,\mathbb{Z})$ is finite. Then for any $m\geq 1$, $CW_{\bz{p^m}}B\F$ is the homotopy fibre of the rationalization of $B\F$.

\end{prop}

\begin{proof}

The proof of Proposition \ref{p:G simple} shows that $Cl_{\F(G)}(\bz{p^m})=S$.
Hence, $Cl_{\F}(\bz{p^m})=S$, and in particular $\ker (\p{r}{p}) = S$.
By Proposition \ref{p:Linp+H2pm} and Remark \ref{r:1-connected}, $B\F$ is $CW_{\bz{p^m}}$-good. Then, by Proposition \ref{p:CW phi} the augmentation $CW_(\bz{p^m})B\F\rightarrow B\F$ is a mod $p$-equivalence. As $B\F$ is 1-connected and $(CW_{\bz{p^m}} (\Linp))_{\Q} \simeq \ast$ and $\p{CW_{\bz{p^m}} (\Linp)}{q} \simeq \ast$ for $q \neq p$, the result follows from a Sullivan arithmetic square,.

\end{proof}

Let us describe a concrete example of this situation.

 \begin{example}
  Consider the space $BDI(4)$ constructed by Dwyer-Wilkerson in \cite{MR1274096} which
  is the classifying space of a $2$-compact group. Its mod $2$ cohomology are the Dickson invariants of rank $4$, $H^{\ast}((\bz{2})^4;\mathbb{F}_2)^{GL_4 \mathbb{F}_2}$,
  and the $2$-adic integral cohomology, computed in \cite{OsseSuter}, is $$H^*(BDI(4);\mathbb Z^\wedge_2)\cong \mathbb Z^\wedge_2[v_8,v_{12},v_{28},\tau_{15}]/(2\tau_{15}).$$ In particular,
  $BDI(4)$ is $2$-connected, and $H_2(BDI(4);\mathbb Z)$ is trivial.
  There is a map of $2$-compact groups $BSpin(7) \rightarrow BDI(4)$, and both groups share the same maximal $2$-toral subgroup.

  In \cite{MR1943386}, Levi-Oliver construct for every odd prime $q$ a family of exotic $2$-local finite groups, which are given by the exotic Solomon fusion systems $Sol(q)$ over the $2$-Sylow $S$ of $Spin(q)$. In particular, $Sol(q)$ contains the fusion system $\F_{Spin_7(q)}(S)$ of $Spin_7(q)$, which is the centralizer of a central element in $S$. These authors also show that $BDI(4)$ is homotopy
  equivalent to the $2$-completion of the union of spaces $BSol(q^n)$, and this identifies $\F_{Spin(7)}(S)$ as a subsystem of $\F_{DI(4)}(S)$ over the same Sylow $2$-subgroup. Hence, by the previous result, we have a homotopy fibration:

  $$
   CW_{\bz{2^m}}(BDI(4)) \rightarrow BDI(4) \rightarrow (BDI(4))_{\Q}.
  $$

 \end{example}

A similar situation can be found with the $p$-fusion of $SU(p)$ inside the $p$-fusion of some of the Aguad\'e $p$-compact groups defined in \cite{Aguade}.

}

{\color{black}

 At this point, the main question that remain unanswered in our context, and which our methods are unable to solve, is the following:

 \begin{question}
  If $G$ is a compact Lie group (not necessarily connected nor simple), what can we say about the $\bz{p^m}$-cellularization of $\p{BG}{p}$ if $BS$ is not $\bz{p^m}$-cellular?
 \end{question}


\bibliographystyle{alpha}
\bibliography{bibliografia}

\bigskip
{\small
\begin{minipage}[t]{8 cm}
Nat\`{a}lia Castellana and Alberto Gavira-Romero\\ Departament de Matem\`atiques,\\
Universitat Aut\`onoma de Barcelona,\\ E-08193 Bellaterra, Spain
\\\textit{E-mail:}\texttt{\,natalia@mat.uab.es}, \\
\phantom{\textit{E-mail:}}\texttt{\,\,albertogavira@gmail.com}
\end{minipage}
}
\vspace{1cm}

{\small
\begin{minipage}[t]{8 cm}
Ram\'{o}n Flores\\ Departamento de Geometr\'{i}a y Topolog\'{i}a,\\
Universidad de Sevilla-IMUS,\\ E-41012 Sevilla, Spain
\\\textit{E-mail:}\texttt{\,ramonjflores@us.es} \\

\end{minipage}
}

}
\end{document}